\documentclass[a4paper,twoside,10pt]{article}
\usepackage[a4paper,left=2.4cm,right=2.4cm, top=2.5cm, bottom=2.5cm]{geometry}
\usepackage{cuted}

\usepackage{xcolor}
\definecolor{darkblue}{rgb}{0.0, 0.0, 0.55}

\usepackage{amsmath}
\usepackage{amssymb}
\usepackage{amsthm}
\usepackage{amscd}
\usepackage{mathrsfs}
\usepackage{mathtools}
\usepackage{bbm}
\usepackage{stmaryrd}
\SetSymbolFont{stmry}{bold}{U}{stmry}{m}{n}
\usepackage{cancel}
\allowdisplaybreaks

\usepackage{graphicx}
\usepackage[font=footnotesize,labelfont=bf]{caption}
\usepackage{subcaption}
\usepackage{float}
\usepackage{pgfplots}
\pgfplotsset{compat=1.18}

\usepackage{array}
\usepackage{multirow}
\usepackage{booktabs}
\usepackage{blkarray}
\usepackage{pgfplotstable} 

\usepackage{enumitem}
\usepackage{cite}

\usepackage{orcidlink}
\usepackage{hyperref}

\newcommand{\jump}[1]{\llbracket #1 \rrbracket}
\newcommand{\av}[1]{\{\!\!\{#1\}\!\!\}}

\numberwithin{equation}{section}
\hypersetup{
    colorlinks,
    linktoc=section,
    citecolor=red,
    linkcolor=darkblue,
    urlcolor=magenta,
    pdfborder={0 0 0}
}

\newtheorem{theorem}{Theorem}[section]
\newtheorem{lemma}[theorem]{Lemma}
\newtheorem{proposition}[theorem]{Proposition}
\newtheorem{assumption}{Assumption}[section]
\newtheorem{corollary}[theorem]{Corollary}
\newtheorem{remark}[theorem]{Remark}
\newtheorem{definition}[theorem]{Definition}
\newcommand{\eremk}{\hbox{}\hfill\rule{0.8ex}{0.8ex}}

\newcommand{\bO}{\boldsymbol{0}}
\newcommand{\R}{\mathbb{R}}
\newcommand{\N}{\mathbb{N}}
\newcommand{\bx}{\mathbf{x}}
\newcommand{\by}{\mathbf{y}}
\newcommand{\bw}{\mathbf{w}}
\renewcommand{\H}{\boldsymbol{H}}
\renewcommand{\L}{\boldsymbol{L}}

\newcommand{\W}{\boldsymbol{\boldsymbol{W}}}

\newcommand{\dx}{\,\mathrm{d}\bx}
\newcommand{\dy}{\,\mathrm{d}\by}

\newcommand{\dt}{\,\mathrm{d}t}

\newcommand{\Norm}[2]{\|#1\|_{#2}}
\newcommand{\semiNorm}[2]{| #1 |_{#2}}

\newcommand{\calD}{\mathcal{D}}
\newcommand{\calH}{\mathcal{H}}
\newcommand{\QT}{Q_T}
\newcommand{\Deltax}{\Delta_{\bx}}
\newcommand{\nablax}{\nabla_{\bx}}
\newcommand{\Nablax}{\boldsymbol{\nabla}_{\bx}}
\newcommand{\Nablay}{\boldsymbol{\nabla}_{\by}}
\newcommand{\divx}{\mathrm{div}_{\bx}}
\newcommand{\divy}{\mathrm{div}_{\by}}

\newcommand{\A}{\mathcal{A}}
\newcommand{\At}{\mathcal{A}_t}
\renewcommand{\u}{\boldsymbol{u}}
\renewcommand{\v}{\boldsymbol{v}}
\newcommand{\z}{\boldsymbol{z}}
\newcommand{\bvphi}{\boldsymbol{\varphi}}
\newcommand{\hu}{\widehat{\boldsymbol{u}}}
\newcommand{\hv}{\widehat{\boldsymbol{v}}}
\newcommand{\f}{\boldsymbol{f}}
\newcommand{\nF}{\boldsymbol{n}_F}

\newcommand{\dpt}{\partial_t}
\newcommand{\ddt}{\frac{\mathrm{d}}{\mathrm{d}t}}
\newcommand{\Dt}{D_t}

\DeclareMathOperator{\ddiv}{div}

\DeclareMathOperator{\esssup}{ess\,sup}
\DeclareMathOperator{\essinf}{ess\,inf}
\newcommand{\CPF}{C_{\mathrm{PF}}}
\DeclareMathOperator{\tr}{tr}

\newcommand{\ThO}{\widehat{\mathcal{T}}_h}
\newcommand{\Tht}{{\mathcal{T}}_h(t)}
\newcommand{\Tt}{\mathcal{T}_{\tau}}
\newcommand{\Kh}{\widehat{K}}
\newcommand{\diam}{\mathrm{diam}}
\newcommand{\hK}{h_K}
\newcommand{\hF}{h_F}
\newcommand{\Fh}{\widehat{\mathcal{F}}_h}
\newcommand{\Fho}{\widehat{\mathcal{F}}_h^{\mathcal{I}}}
\newcommand{\Fhb}{\widehat{\mathcal{F}}_h^{\partial}}
\newcommand{\Qn}{\mathcal{Q}_n}
\newcommand{\Qj}{\mathcal{Q}_j}
\newcommand{\In}{I_n}
\newcommand{\tnmo}{t_{n-1}}
\newcommand{\tn}{t_n}
\newcommand{\Ft}{F_t}
\newcommand{\Kt}{K_t}
\newcommand{\calK}{\mathcal{K}}
\newcommand{\nFt}{\boldsymbol{n}_{\Ft}}

\newcommand{\Vhk}{\widehat{\boldsymbol{\mathcal{V}}}_h^k}

\newcommand{\Mhk}{\widetilde{\mathcal{M}}_h^{k-1}}
\newcommand{\Vht}{\boldsymbol{\mathcal{V}}_{h\tau}}
\newcommand{\Zht}{\boldsymbol{\mathcal{Z}}_{h\tau}}
\newcommand{\Mht}{{\mathcal{M}}_{h\tau}}

\newcommand{\ba}{\boldsymbol{\widehat{\alpha}}}
\newcommand{\hb}{\widehat{\beta}}
\newcommand{\uht}{\u_{h\tau}}
\newcommand{\vht}{\v_{h\tau}}
\newcommand{\zht}{\z_{h\tau}}
\newcommand{\vh}{\v_h}

\newcommand{\vvht}{\boldsymbol{\varphi}_{h\tau}}
\newcommand{\pht}{p_{h\tau}}
\newcommand{\qht}{q_{h\tau}}
\newcommand{\Pp}[2]{\mathbb{P}^{#1}(#2)}

\newcommand{\ah}{a_h}
\newcommand{\aht}{a_{h,\tau}}

\newcommand{\cht}{c_{h,\tau}}
\newcommand{\Bht}{\mathcal{B}_{h,\tau}}

\newcommand{\Id}{\mathrm{Id}}

\newcommand{\Piht}{\Pi_{h\tau}}

\newcommand{\eu}{\boldsymbol{e}_{\u}}
\newcommand{\epi}{\boldsymbol{e}_{\u}^{\Pi}}
\newcommand{\Pt}{\mathcal{P}_{\tau}}

\newcommand{\IBDM}{\widehat{\mathcal{I}}_h^{\mathrm{BDM}}}
\newcommand{\PihM}{\Pi_{h}^{\mathcal{M}}}

\title{Pressure-robust ALE space--time DG method for the\\
Stokes equations on moving domains} 
\author{L. Beir\~ao da Veiga\thanks{Department of Mathematics and Applications, University of Milano-Bicocca, Via Cozzi 55, 20125 Milan, Italy (\href{mailto: lourenco.beirao@unimib.it}{lourenco.beirao@unimib.it}, \href{mailto:sergio.gomezmacias@unimib.it}{sergio.gomezmacias@unimib.it}, \href{mailto:k.haile@campus.unimib.it}{k.haile@campus.unimib.it})} \thanks{IMATI-CNR ``E. Magenes", Via Ferrata 5, 27100 Pavia, Italy}\ \orcidlink{0000-0001-5895-469X} \and S. G\'omez\footnotemark[1] \footnotemark[2]\ \orcidlink{0000-0001-9156-5135} \and K. B. Haile\footnotemark[1]\ \orcidlink{0009-0007-2941-0926}
}
\date{}

\begin{document}

\maketitle
\begin{abstract}
    \noindent 
    We propose and analyze a space--time discontinuous Galerkin method for the incompressible Stokes equations on moving domains within the arbitrary Lagrangian--Eulerian setting.
    We use a contravariant Piola map in the definition of the discrete velocity space to preserve the pointwise divergence-free property on the discrete level. We show that the method is inf--sup stable, with no constraints on the spatial mesh or the time partition. We also establish a priori error estimates in the energy norm for arbitrary degrees of approximation in space and time. For piecewise-constant and piecewise-linear approximations in time, we show that the method is also robust at low viscosity regimes, and provide numerical evidence suggesting that this property extends to high-order cases as well. 
    We present several numerical experiments to validate our theoretical findings.
\end{abstract}

\section{Introduction}\label{sec:intro}
This work concerns the design and analysis of a space--time discontinuous Galerkin method for the  incompressible Stokes equations on moving domains, which preserves the exact divergence-free property for the velocity.
\paragraph{The ALE setting for moving domains.}
Let~$T > 0$ be a prescribed final time. For~$d \in \{2, 3\}$, we consider a \emph{moving domain} as an indexed family~$\{\Omega(t)\}_{t \in [0, T]}$, where~$\Omega(t) \subset \R^d$ is an open bounded domain, for each~$t \in [0, T]$. We assume (only for simplicity, see Remarks~\ref{rem:geo-exact} and~\ref{rem:initial} for natural generalizations) that the initial domain $\Omega_0 := \Omega(0)$ is an open polytope that also serves as our \emph{reference domain}.

Following the ALE approach, we assume that the boundary deformation is determined by a velocity~$\bw$, which can be extended with standard techniques to  also dictate the bulk deformation. For each~$t \in [0, T]$, such an extension induces a bijective map~$\At : \Omega_0 \to \Omega(t)$. Moreover, the family of maps~$\{\At\}_{t \in [0, T]}$ can be used to define the space--time function~$\A(\by, t) := \At(\by)$, which satisfies
\begin{equation}\label{eq:w-x-identity}
    \bw(\bx, t) = {\dpt} \A (\by, t) \quad \forall \by \in \Omega_0, \text{ with } \bx = \At(\by) .
\end{equation}
In what follows, we will consistently use~$\bx$ for the physical spatial coordinate in~$\Omega(t)$ and~$\by$ for the coordinate in the reference domain~$\Omega_0$. This convention is also used to stress the variable of differentiation for differential operators. 

The space--time domain is defined as
\begin{equation*}
    \QT := \big\{(\bx, t) \in \R^d \times \R^+ \ : \ t \in (0, T) \text{ and } \bx = \At(\by), \text{ for some } \by \in \Omega_0 \big\},
\end{equation*}
or, equivalently, $\QT$ is the image of the tensor-product domain~$\Omega_0 \times (0, T)$ under the extended space--time mapping $(\by, t) \mapsto (\A(\by, t), t)$.
\paragraph{Model problem.} 
Given a constant viscosity~$\nu > 0$, a source term~$\f : \QT \to \R^d$, and an initial datum~$\u_0 : \Omega_0 \to \R^d$, 
we consider the following time-dependent Stokes equations: 
find the velocity~$\u : \QT \to \R^d$ and the (minus) pressure~$p : \QT \to \R$ such that
\begin{subequations}\label{eq:model-problem}
    \begin{alignat}{3}
        \dpt \u - \nu \Deltax \u - \nablax p & = \f & & \quad \text{in~$\QT$}, \\
        \divx \u & = 0 & & \quad \text{in~$\QT$}, \\
        \u & = \bO & & \quad \text{on~$\partial\Omega(t) \times \{t\}, \ t \in [0, T]$,} 
        \label{eq:model-problem-bc} \\
        \u(\cdot, 0) & = \u_0(\cdot) & & \quad \text{in~$\Omega_0$}.
    \end{alignat}
\end{subequations}
In the ALE framework, model~\eqref{eq:model-problem} can be rewritten using the~\emph{material derivative}, which corresponds to the time derivative along the characteristic lines of the ALE map~$\A$. For a sufficiently smooth function~$\boldsymbol{\varphi}$, it is defined, via the chain rule, as
\begin{equation*}
    \Dt \boldsymbol{\varphi}(\bx, t) := \ddt \boldsymbol{\varphi}(\At(\by), t) = \dpt \boldsymbol{\varphi}(\bx, t) + \Nablax \boldsymbol{\varphi} (\bx, t) \bw(\bx, t) \, , \quad \bx = \At(\by),
\end{equation*}
where~$\Nablax$ denotes the spatial gradient operator acting componentwise. The resulting reformulation of~\eqref{eq:model-problem} reads: 
find the velocity~$\u : \QT \to \R^d$ and the (minus) pressure~$p : \QT \to \R$ such that
\begin{subequations}\label{eq:model-problem-ALE}
    \begin{alignat}{3}
        \Dt \u - (\Nablax \u) \bw - \nu \Deltax \u - \nablax p & = \f & & \quad \text{in~$\QT$}, \\
        \divx \u & = 0 & & \quad \text{in~$\QT$}, \\
        \u & = \bO & & \quad \text{on~$\partial\Omega(t) \times \{t\}, \ t \in [0, T]$,}\\
        \u(\cdot, 0) & = \u_0(\cdot) & & \quad \text{in~$\Omega_0$}.
    \end{alignat}
\end{subequations}
The specific condition \eqref{eq:model-problem-bc} is included only for simplicity; the method and analysis in the present contribution can be easily extended to different boundary conditions.
\paragraph{Previous works.}
Several approaches have been proposed in the literature for the numerical approximation of PDEs on moving domains within the ALE framework. Below, we briefly describe some of the main ideas employed in the design of high-order space--time methods.
\begin{enumerate}[label = {\emph{\roman*)}}, ref = {\emph{\roman*)}}, topsep = 0pt, itemsep = 0pt]
    \item A first approach consist of approximating the domain at each time level with a standard mesh, and subsequently constructing space--time elements using linear interpolation of~$\A$ in time for internal elements and high-order interpolation for curved elements on the boundary. The discrete spaces are then defined through isoparametric mappings from a reference element; see, e.g., \cite{Van_der_Vegt_Van_der_Ven:2002,Rhebergen_Cockburn:2013,Rhebergen_Cockburn:2012}. 
    \item An alternative approach is to directly decompose the space--time domain with $(d+1)$-dimensional simplices or polytopes and treat the time variable as an additional spatial dimension; see~\cite{Horvath_Rhebergen:2020,Horvath_Rhebergen:2019} for hybridizable discontinuous Galerkin methods, and the review~\cite{Gaburro:2021} on finite volume and discontinuous Galerkin (DG) schemes. 
    \item One can also formulate the problem in the reference coordinates, and then apply standard discretizations for tensor-product space--time domains; see~\cite{Persson_Peraire:2009}.
    \item \label{approach-iv} The high-order DG-in-time approach of~\cite{Bonito_Kyza_Nochetto:2013b,Bonito_Kyza_Nochetto:2013} extends the ideas of~\cite{Formaggia_Nobile:1999,Formaggia_Nobile:2004} to high-order approximations in time. 
    This approach has also been applied  in~\cite{Balzasov_Feistauer:2015,Balazsov_Feistauer_Vlasak:2018} to nonlinear convection--diffusion equations.
    In this framework, the discrete space at each time level is obtained by transporting the reference basis functions through the ALE map~$\A$. In particular, the spaces considered in~\cite{Bonito_Kyza_Nochetto:2013,Bonito_Kyza_Nochetto:2013b} are piecewise polynomials in time along the trajectories induced by~$\A$, and can be combined with a conforming finite element space defined on the reference mesh.
\end{enumerate}
We refer to~\cite[\S1]{Rao_Wang_Xie:2025} for a recent comprehensive overview of stability results and a priori error estimates for finite element methods in the ALE setting.
\paragraph{A specific challenge.}
A major challenge in the fully discrete (i.e., discrete in \emph{both} space and time) approximation of incompressible fluid problems in moving domains is ensuring the solenoidal constraint and pressure-robustness. 
Indeed, this property is intimately related to employing a discrete velocity/pressure pair that satisfies a De Rham diagram (see, e.g., \cite[\S4.3 and 4.4]{john2017divergence}). 
When considering a high-order method with exact (or high-order) geometric approximation, achieving pressure-robustness essentially requires using a Piola map to define the velocity space in the physical domain. For general geometries, this is often required also for standard finite elements, as many inf--sup stable pairs degenerate under non-affine mappings unless a Piola transform is used to define the velocity space. 
On the other hand, incorporating a Piola map introduces a very strong tangling of space and time, posing additional difficulties in the analysis of the method. For instance, unlike what happens for a standard push-forward, the presence of the Piola map implies that the material derivative of a ``polynomial in time'' discrete velocity is not a function of the same kind with one degree lower. 
This prevents the use of many standard arguments used in the literature for the analysis of DG time discretizations; cf. \cite{Gomez:2026b}.
\paragraph{Present contribution.}
We propose a space--time DG method for the time-dependent Stokes equations on moving domains. 
The proposed method is closely related to approach~\ref{approach-iv}, but (in addition to addressing also the space discretization) employs the controvariant Piola map associated with~$\A$ to transport the discrete spaces. Combined with the use of~$\H(\mathrm{div})$-conforming velocity spaces \cite{Guzman_Shu_Sequeira:2017,schroeder2018divergence,barrenechea_etal:2020,Han_Hou:2021,da2025reynolds} and a compatible pressure discretization in the reference domain, this construction preserves the divergence structure under the ALE mapping and allows the divergence-free constraint to be satisfied exactly (up to quadrature errors) at all times. Furthermore, in this approach, %
the geometry of the domain is approximated exactly (see Remarks \ref{rem:initial} and \ref{rem:geo-exact}).

The main advancements in this work are:
\begin{itemize}
    \item We derive stability bounds and establish a priori error estimates for arbitrary degrees of approximation in both space and time. Moreover, we do not require smallness conditions on the time steps, nor do we require (global or local) quasi-uniformity of the spatial and temporal partitions.
    \item The method is pressure-robust, meaning that modifications of the data which only affect pressure at the continuous level will maintain such a property also for the discrete solution. As a consequence, velocity errors are not influenced by pressure errors, which renders the velocity approximation free from pressure-induced pollution. 
    \item For the low-order cases (with approximations in time of degree~$\ell = 0$ or~$\ell = 1$), we show that the method is robust with respect to small values of the %
    viscosity parameter. Preliminary numerical experiments suggest that this property holds also for high-order approximations in time. 
\end{itemize}
In addition to the theoretical derivations, we present numerical tests in $(2+1)$ dimensions which support our theoretical findings. These include verifying the expected convergence orders in both space and time, assessing the pressure-robustness, and showing uniformly bounded error for small values of the viscosity parameter~$\nu$.
\paragraph{Structure of the article.}
The remainder of this work is organized as follows. 
Section~\ref{sec:2} introduces some basic notation and provides a set of preliminary results for the mappings involved. We then present the space--time method in Section~\ref{sec:3}. The stability and convergence analyses are developed in Sections~\ref{sec:4} and~\ref{sec:5}, respectively. In Section~\ref{sec:6}, we establish an improved result valid for the low-order cases in time. Finally, the numerical experiments are presented in Section~\ref{sec:numerical-tests}.

\section{Notation and maps}\label{sec:2}
We use standard notation for~$L^p$ and Sobolev spaces, as well as for Bochner spaces of functions defined in the tensor-product domain~$\Omega_0 \times (0, T)$. Moreover, we will introduce some suitable spaces for functions defined in the space--time domain~$\QT$. 

For a bounded domain~$\calD \subset \R^{d}$ ($d \in \{1, 2, 3\}$) with Lipschitz boundary~$\partial \calD$, $m \in \N$, and~$p \in [1, \infty]$, we denote by~$W^{m,p}(\calD)$ the corresponding Sobolev space on~$\calD$ equipped with the standard seminorm~$\semiNorm{\cdot}{W^{m,p}(\calD)}$ and norm $\Norm{\cdot}{W^{m,p}(\calD)}$.
For~$m=0$, we obtain the Lebesgue space~$L^{p}(\calD) := W^{0,p}(\calD)$, and, in particular, $L^2(\calD)$ is the space of Lebesgue square-integrable functions with inner product~$(\cdot, \cdot)_{\calD}$ and norm~$\Norm{\cdot}{L^2(\calD)}$. 
For~$p=2$, we have the Hilbert space~$H^{m}(\calD) := W^{m,2}(\calD)$, with~$H_0^1(\calD)$ denoting the subspace of functions in~$H^1(\calD)$ with zero trace on~$\partial \calD$. 
We use boldface to denote spaces of $d$-vector-valued functions, as well as their elements.
Moreover, we denote by~$\H(\ddiv; \calD)$ the subspace of functions in~$\L^2(\calD)$ with divergence in~$L^2(\calD)$.

Given a time interval~$(a, b)$, $p \in [1, \infty]$, $m \in \N$, and a separable Banach space~$(Z, \Norm{\cdot}{Z})$, the corresponding Bochner spaces are given by
\begin{alignat*}{3}
    L^p(a, b; Z) & := \Big\{v : (a, b) \to  Z \ : \ v \ \text{measurable and} \  \Norm{v}{L^p(a, b; Z)} < \infty\ \Big\}, \\
    W^{m, p}(a, b; Z) &: = \Big\{v \in L^p(a, b; Z) \ : \ \partial_t^{(i)} v \in L^p(a, b; Z), \ \text{for } i =1, \ldots, m \Big\},
\end{alignat*}
where
\begin{equation*}
    \Norm{v}{L^p(a, b; Z)} :=
    \begin{cases}
        \displaystyle \Big(\int_a^b \Norm{v(t)}{Z}^p  \dt\Big)^{1/p} & \text{ if } p \in [1, \infty), \\
        \esssup_{t \in (a, b)} \Norm{v(t)}{Z} & \text{ if } p = \infty.
    \end{cases}
\end{equation*}
For functions defined in~$\QT$, let~$\{Y(t)\}_{t\in [0, T]}$ denote a family of Sobolev spaces, where~$Y(t)$ is a space of functions defined in~$\Omega(t)$ for~$t \in [0, T]$ (e.g., $Y(t) = L^2(\Omega(t))$ or~$Y(t) = H_0^1(\Omega(t))$). We define
\begin{alignat*}{3}
    L^p(Y; \QT) & := \Big\{v \ : \ v(t)  \in Y(t) \text{ for a.e. $t \in [0, T]$} \text{ and } \Big(\int_0^T \Norm{v(\cdot, t)}{Y(\Omega(t))}^p \dt \Big)^{1/p} < \infty \Big\}, \quad  p \in [1, \infty), \\
    L^{\infty}(Y; \QT) &:= \Big\{v \ : \ v(t)  \in Y(t) \text{ for a.e. $t \in [0, T]$} \text{ and }  \esssup_{t \in (0, T)} \Norm{v(\cdot, t)}{Y(\Omega(t))} < \infty \Big\}, \\
    W^{m, p}(Y; \QT) & := \Big\{v \in L^p(Y; \QT) \ : \ \dpt^{(i)} v \in L^p(Y; \QT), \ i = 0, \ldots, m \Big\},  \hspace{0.9in}  p \in [1, \infty], \ m \in \N ,
\end{alignat*}
with the obvious corresponding norms. We will conveniently use the notation~$H^m(Y; \QT) := W^{m, 2}(Y; \QT)$.

\smallskip\noindent
The derivations of the present work hold under the following assumption on the map~$\At$.
\begin{assumption}[Regularity of~$\At$]\label{asm:At}
    Recalling that~$\A$ is the space--time function induced by the family of maps~$\{\At\}_{t \in [0, T]}$, we assume
    \begin{enumerate}[label = {(\roman*)}, ref = {(\roman*)}]
        \item \label{asm:At-1} $\A \in W^{1, \infty}(0, T; \W^{2, \infty}(\Omega_0))$;
        \item \label{asm:At-2} there exists a positive constant~$c_{\A}$ such that, for all~$t \in [0, T]$,
        \begin{equation*}
            \Norm{\At(y_1) - \At(y_2)}{\infty} \geq c_{\A} \Norm{y_1 - y_2}{\infty} \quad \forall y_1, \, y_2 \in \Omega_0,
        \end{equation*}
        where~$\Norm{\cdot}{\infty}$ denotes the~$\ell_{\infty}$ norm of~$d$-vectors.
    \end{enumerate}
\end{assumption}
\begin{remark}[Consequences of Assumption~\ref{asm:At}] \label{rem:assumptions-At}
    Assumption~\ref{asm:At} implies several fundamental properties:
    \begin{itemize}[itemsep = 0pt, topsep = 1.5pt]
        \item Assumption~\ref{asm:At}\ref{asm:At-1} and the continuous embedding~$W^{1, \infty}(0, T; \W^{2, \infty}(\Omega_0)) \hookrightarrow C^{0,1}([0, T]; \boldsymbol{C}^{1,1}(\overline{\Omega_0}))$ guarantees that, for each~$t \in [0,T]$, both~$\At$ and its spatial Jacobian matrix~$\Nablay \At$ are Lipschitz continuous on~$\overline{\Omega_0}$, with uniform Lipschitz constant for all~$t \in [0, T]$. Furthermore,~$\bw \in L^{\infty}(\boldsymbol{W}^{1,\infty},\QT)$.
        \item The uniform strong injectivity condition in Assumption~\ref{asm:At}\ref{asm:At-2} ensures that the inverse mapping~$\At^{-1} : \Omega(t) \to \Omega_0$ exists for all~$t \in [0,T]$ and is Lipschitz continuous, with uniform Lipschitz constant~$c_{\A}^{-1}$.
        \item The combination of the previous two points implies that there exists a positive constant~$\mu_{\star}$ such that~$\det \Nablay \At(\by) \geq \mu_{\star} > 0$ for all~$(\by, t) \in \Omega_0 \times [0, T]$.
        \item Since~$\overline{\Omega_0} \times [0, T]$ is a compact set, the regularity of~$\A$ implies that there exists a fixed, open, bounded domain~$\Omega \subset \mathbb{R}^d$ such that~$\Omega(t) \subset \Omega$ for all~$t \in [0, T]$.
    \end{itemize}
    For further details, we refer to the discussion in~\cite[\S2.1]{Bonito_Kyza_Nochetto:2013}.
    \eremk
\end{remark}
\subsection{The Piola transform}
To relate the Sobolev spaces in the reference domain~$\Omega_0$ to those in~$\Omega(t)$, while preserving divergence-free functions, we define the following (contravariant) Piola transform (see~\cite[Def.~2.7]{Rognes_Kirby_Logg:2009}).
\begin{definition}[Piola transform]
    For each~$t \in [0, T]$, let~$J_t(\by) := \Nablay \At(\by)$ be the spatial Jacobian of the ALE mapping. 
    The fixed-time Piola transformation~$\phi_t : \boldsymbol{L}^2(\Omega_0) \to \boldsymbol{L}^2(\Omega(t))$ is defined for any reference vector field~$\widehat{\boldsymbol{\varphi}} \in \boldsymbol{L}^2(\Omega_0)$ as
    \begin{equation*}
        (\phi_t \widehat{\boldsymbol{\varphi}})(\bx) := \frac{1}{\det J_t(\by)} J_t(\by) \, \widehat{\boldsymbol{\varphi}}(\by) \, , \quad \text{with } \bx = \At(\by).
    \end{equation*}
    The inverse fixed-time Piola transformation~$\phi_t^{-1} : \boldsymbol{L}^2(\Omega(t)) \to \boldsymbol{L}^2(\Omega_0)$ maps a vector field~$\boldsymbol{\varphi} \in \boldsymbol{L}^2(\Omega(t))$ back to the reference domain as:
    \begin{equation*}
        (\phi_t^{-1} \boldsymbol{\varphi})(\by) := \det J_t(\by) J_t^{-1}(\by) \, \boldsymbol{\varphi}(\bx), \quad \text{with } \bx = \At(\by).
    \end{equation*}
    Pointwise extension in time defines the global space--time Piola operator~$\phi: H^1(0,T; \boldsymbol{L}^2(\Omega_0)) \to H^1(\boldsymbol{L}^2; \QT)$ such that~$(\phi \widehat{\boldsymbol{\varphi}})(\bx, t) := [ \phi_t (\widehat{\boldsymbol{\varphi}}(\cdot, t) )](\bx)$.
    Accordingly, the global inverse space--time Piola operator~$\phi^{-1}: H^1(\boldsymbol{L}^2; \QT) \to H^1(0,T; \boldsymbol{L}^2(\Omega_0))$ is defined as~$(\phi^{-1} \boldsymbol{\varphi})(\by, t) := [ \phi_t^{-1} (\boldsymbol{\varphi}(\cdot, t) )](\by)$. Such definitions can be generalized in a natural way to functions that are only piecewise~$H^1$ in time.
\end{definition}
The key result related to this Piola transform is given in the next Lemma; see~\cite[Ex.~4.6]{Rognes_Kirby_Logg:2009} and~\cite[Lemma~2.8]{Djurdjevac_Graser_Herbert:2023}).
\begin{lemma}[Divergence preservation]\label{lemma:div-preservation}
    For each~$t \in [0, T]$ and all~$\hv \in \H(\divy; \Omega_0)$, the following identity holds:
    \begin{equation*}
        \divx (\phi_t \hv) (\bx) = \frac{1}{\det J_t(\by)} \divy \hv (\by)  \quad \text{for a.e. } \bx = \At (\by) \in \Omega(t).
    \end{equation*}
\end{lemma}
In addition, the following identity will be used in the convergence analysis.
\begin{lemma}[Time derivative of the Piola transform]\label{lemma:Dt-phi}
    For all~$\hv \in H^1(0,T; \L^2(\Omega_0))$, the space--time Piola transform~$\phi \hv$ satisfies the following identity:
    \begin{equation*}
        \Dt (\phi \widehat{\v}) = \phi (\partial_t \widehat{\v}) + (\Nablax \bw - \divx \bw \mathbb{I}_d) \phi \widehat{\v} \, \qquad \text{in $L^2(\boldsymbol{L}^2; \QT)$},
    \end{equation*}
    where~$\mathbb{I}_d$ is the $d \times d$ identity matrix.
\end{lemma}
\begin{proof}
    Using the definition of the Piola transform~$\phi_t$, the fact that the material derivative~$\Dt$ is the partial derivative with respect to~$t$ for fixed~$\by$, and the Leibniz rule, for~$\bx = \At(\by)$, we can write
    \begin{align}
        \nonumber
        [\Dt (\phi \widehat{\v})](\bx, t) & = \dpt \Big(\frac{1}{\det J_t(\by)} J_t(\by) \widehat{\v} (\by, t) \Big) \\
        \label{eq:aux-identity-Dt-phi_t}
        & = \dpt \Big(\frac{1}{\det J_t(\by)}\Big) J_t(\by) \widehat{\v}(\by, t) + \frac{1}{\det J_t(\by)}  \partial_t (J_t(\by)) \widehat{\v}(\by, t) + \frac{1}{\det J_t(\by)} J_t(\by) \dpt \widehat{\v}(\by, t).
    \end{align}
    Elementary calculus operations, identity~\eqref{eq:w-x-identity}, and the chain rule give
    \begin{equation}\label{eq:dpt-Jt}
        \partial_t (J_t(\by))  =  \dpt (\Nablay A_t(\by)) = \Nablay (\dpt \bx) = \Nablay \bw = (\Nablax \bw ) J_t(\by).
    \end{equation}
    Moreover, using the Jacobi formula (see~\cite[Thm. 8.1 in Ch. 8, Part III]{Magnus_Neudecker:2019}), identity~\eqref{eq:dpt-Jt}, and the commutativity of the matrix product under the trace operator, we obtain
    \begin{equation*}
        \dpt (\det J_t(\by)) = \det J_t(\by) \tr (J_t(\by)^{-1} \dpt J_t(\by))  = \det J_t(\by) \tr(J_t(\by)^{-1} \Nablax \bw J_t(\by)) = \det J_t(\by) \divx \bw,
    \end{equation*}
    which implies
    \begin{equation}\label{eq:dpt-inv-det-Jt}
        \dpt \Big(\frac{1}{\det J_t(\by)} \Big) = - \frac{1}{(\det J_t(\by))^2} \dpt (\det J_t(\by)) = -\frac{1}{\det J_t(\by)} \divx \bw. 
    \end{equation}
    Inserting identities~\eqref{eq:dpt-Jt} and~\eqref{eq:dpt-inv-det-Jt} into~\eqref{eq:aux-identity-Dt-phi_t}, we get
    \begin{align*}
        [\Dt (\phi \widehat{\v})](\bx, t) &= - \divx \bw(\bx, t) \left( \frac{1}{\det J_t(\by)} J_t(\by) \widehat{\v}(\by, t) \right) + \Nablax \bw(\bx, t) \left( \frac{1}{\det J_t(\by)} J_t(\by) \widehat{\v}(\by, t) \right) \\
        & \quad + \frac{1}{\det J_t(\by)} J_t(\by) \dpt \widehat{\v}(\by, t) \\
        &= [\phi (\partial_t \widehat{\v})](\bx, t) - \divx \bw(\bx, t) (\phi \widehat{\v})(\bx, t) + \Nablax \bw(\bx, t) (\phi \widehat{\v})(\bx, t) \qquad \text{ for a.e.~$(\bx, t) \in \QT$,}
    \end{align*}
    which yields the desired global identity in $L^2(\boldsymbol{L}^2; \QT)$.
\end{proof}
We end this section with the following key notation remark.
\begin{remark}[Hat and tilde notation]\label{rem:notation_conventions}
    Any scalar-valued function~$\varphi \in L^1(\QT)$  can be uniquely associated with its ALE pull-back map~$\widetilde{\varphi} \in L^1(\Omega_0 \times [0,T])$ via direct composition as~$\widetilde{\varphi}(\by,t) := \varphi(\bx,t)$ for a.e.~$\bx = \At(\by)$.
    Moreover, the ALE push-forward map of a reference function~$\widetilde{\varphi}$ is given by~$\varphi(\bx,t) := \widetilde{\varphi}(\by,t)$ for~a.e.~$\bx = \At (\by)$. 
    This tilde notation can be extended naturally to vector-valued functions.
    The Piola map provides an alternative way to relate vector-valued functions defined in the reference domain~$\Omega_0$ to those defined in~$\Omega(t)$. Throughout the remainder of this work, we adopt a systematic notation to distinguish these types of push-forwards and pull-backs.
    \begin{itemize}
        \item If a function is originally defined in~$\QT$, its tilded (resp. hatted) counterpart denotes its ALE (resp. Piola) pull-back. Analogously, if a tilded (resp. hatted) function is first introduced in the reference domain, the plain symbol without tilde (resp. hat) represents its ALE (resp. Piola) push-forward defined in~$\QT$.
        \item For scalar-valued functions, we systematically employ the tilde~$\widetilde{\varphi}$ notation. For vector-valued functions, both notations can be used depending on the context. The tilde~$\widetilde{\boldsymbol{\varphi}}$ notation denotes the standard pull-back via direct composition, and it preserves pointwise values but not normal fluxes and divergence. 
        Conversely, the hat function~$\widehat{\boldsymbol{\varphi}}$ preserves divergence mappings (see Lemma~\ref{lemma:div-preservation}), which is fundamental for the construction of our discrete velocity spaces. 
        \eremk
    \end{itemize}
\end{remark}
\section{Description of the method}\label{sec:3}
We first describe the mesh notation (Section~\ref{sec:mesh-notation}); we then present the discrete spaces for the velocity and the pressure (Sections~\ref{sec:discrete-spaces} and~\ref{sec:kernel-space}), and finally introduce the proposed space--time method (Section~\ref{sec:DG-method}).
\subsection{Mesh notation}\label{sec:mesh-notation}
Let~$\{\ThO\}_{h > 0}$ be a family of shape-regular, conforming simplicial partitions of the initial domain~$\Omega_0$, where the parameter~$h$ stands for the maximum meshsize of~$\ThO$. The set of facets of~$\ThO$ is denoted by~$\Fh = \Fho \cup \Fhb$, where~$\Fho$ and~$\Fhb$ are the sets of internal and boundary facets, respectively. 
For each~$F \in \Fh$, we set~$\hF := \diam(F)$, and denote by~$\nF$ one of the two unit normal vectors to~$F$, which is fixed once and for all, with the convention that~$\nF$ points outward~$\Omega_0$ if~$F \in \Fhb$. In addition, for~$t \in [0, T]$, we denote by~$\Ft := \At(F)$ the moving facets, and by~$\nFt$ the corresponding unit normal to~$\Ft$ (which does not coincide with~$\At(\nF)$ in general). 
Analogously, for each~$K \in \ThO$ and~$t \in [0, T]$, we define the moving elements $\Kt := \At(K)$, which form the spatial mesh~$\Tht$ of the domain $\Omega(t)$. Associated with this partition, we introduce the following broken Sobolev space:
$$
    H^1(\Tht) := \big\{ q \in L^2(\Omega(t)) \ : \ q|_{\Kt} \in H^1(\Kt) \quad \forall K \in \ThO \big\}.
$$
For~$t \in [0, T]$, we consider the following average and jump operators: for any facet~$\Ft := \At(F)$, with $F$ shared by two elements~$K^+$ and~$K^-$ in~$\ThO$ and~$\nF$ pointing from~$K^+$ to~$K^-$, we define
$$
    \av{q}_{\Ft} := \frac12 \big(q{|_{\Kt^+}} + q{|_{\Kt^-}}\big) \quad \text{ and } \quad \jump{q}_{\Ft} := q{|_{\Kt^+}} - q{|_{\Kt^-}} \qquad \forall q \in H^1(\Tht). 
$$
Analogously, for any boundary facet~$F \in \Fhb$, we set
$$ 
    \av{q}_{\Ft} :=  q{|_{\Ft}} \quad \text{ and } \quad \jump{q}_{\Ft} := q{|_{\Ft}}.
$$
All definitions here above for scalar fields can be seamlessly extended to vector-valued fields.

We consider a partition~$\Tt$ of the time interval~$(0, T)$, determined by the nodes
\begin{equation*}
    0 =: t_0 < t_1 < \ldots < t_N := T.
\end{equation*}
For~$n = 1, \ldots, N$, we define the time interval~$\In := (\tnmo, \tn)$, the time step~$\tau_n := \tn - \tnmo$, and the time-slab
\begin{equation*}
    \Qn := \big\{(\bx, t) \in \R^d \times \R^+ \ : \ t \in \In, \text{ and } \bx = \At(\by), \text{ for some~$\by \in \Omega_0$} \big\}.
\end{equation*}
Similarly, for all~$K \in \ThO$, we define~$\calK_n := \{(\bx, t) \in \R^d \times \R^+ \ : \ t \in \In, \ \text{and } \bx = \At(\by), \ \text{for some~$\by \in K$}\}$.

Furthermore, we define the time-jump $\jump{\cdot}_n$ at $t_n$ for functions~$\v$ with~$\widetilde{\v} \in H^1(\In; \L^2(\Omega_0))$ as follows:
\begin{equation}\label{eq:time-jump}
    \jump{\v}_n (\bx) := \v(\bx, \tn^+) - \v(\bx, \tn^-) \quad \forall \bx = \A_{\tn}(\by) \in \Omega(\tn), 
\end{equation}
where
\begin{equation*}
    \v(\bx, \tn^+) := \lim_{\varepsilon \to 0^+} \widetilde{\v}(\by, \tn + \varepsilon) \quad \text{and} \quad \v(\bx, \tn^-) := \lim_{\varepsilon \to 0^+} \widetilde{\v}(\by, \tn - \varepsilon).
\end{equation*}
\subsection{Discrete spaces}\label{sec:discrete-spaces}
Given a degree of approximation~$k \in \N$ with~$k \geq 1$, we denote by~$\Vhk$ the Brezzi--Douglas--Marini (BDM) space of degree~$k$ defined on~$\ThO$ (see~\cite[\S2.3]{Boffi_Brezzi_Fortin:2013}), with the additional condition of having vanishing normal component on $\partial\Omega_0$. We also denote by~$\Mhk$ the space of piecewise polynomials of degree~$k-1$ with zero mean on~$\Omega_0$.

Given also a degree of approximation in time~$\ell \in \N$, for~$n = 1, \ldots, N$, we denote by~$\Pp{\ell}{\In}$ the space of polynomials of degree at most~$\ell$ defined on~$\In$, and by~$\{\psi_j^{(n)}\}_{j = 0}^{\ell}$ a generic basis for~$\Pp{\ell}{\In}$, where specific choices will be made when convenient.

Furthermore, we define the following discrete spaces on each time-slab~$\Qn$: 
\begin{subequations}
    \begin{alignat}{3}
        \Vht^{(n)} & := \bigg\{\vht : \Qn \to \R^d \ :  \  \vht(\bx, t) = \sum_{j = 0}^{\ell} (\phi_t \ba_{h, j})(\bx) \psi_j^{(n)}(t),\  \text{with}\ \ba_{h,j} \in \Vhk,\ \ \text{for} \  j = 0, \ldots, \ell
        \bigg\}, \label{def:Vht} \\
        \Mht^{(n)} & := \Big\{\qht : \Qn \to \R  \ : \ 
        \qht(\bx, t) = \sum_{j = 0}^{\ell} (\widetilde{\beta}_{h,j} \circ \At^{-1}) (\bx) \psi_j^{(n)}(t), \ \text{with}\ \widetilde{\beta}_{h,j} \in \Mhk , \ \text{ for}\ j = 0, \ldots, \ell\Big\}. \label{def:Mht}
    \end{alignat}
\end{subequations}
We can then define the global discrete spaces as
\begin{subequations}
    \begin{alignat}{3}
        \label{eq:globals:1}
        \Vht & := \Big\{\vht : \QT \to \R^d \ : \ \vht{}{|_{\Qn}} \in \Vht^{(n)}, \text{ for~$n = 1, \ldots, N$} \Big\}, \\
        \label{eq:globals:2}
        \Mht & := \big\{ \qht : \QT \to \R \ : \ \qht {}{|_{\Qn}} \in \Mht^{(n)}, \text{ for~$n = 1, \ldots, N$} \big\}.
    \end{alignat}
\end{subequations}
\begin{remark}[The discrete space~$\Vht$ and its material derivative]
    The discrete space~$\Vht$ can be equivalently defined as~$\Vht : = \phi (\Vhk \otimes \Pp{\ell}{\Tt})$, i.e., any element~$\vht \in \Vht$ is the image under the space--time Piola transform~$\phi$ of a function in the tensor-product space~$\Vhk \otimes \Pp{\ell}{\Tt} $. In contrast to the spaces introduced in~\cite[\S3.1]{Bonito_Kyza_Nochetto:2013}, due to the presence of the (contravariant) Piola transform, the functions in~$\Vht$ are not polynomials in time of degree~$\ell$ along the trajectories defined by~$\At$.
    In particular, using the superscript~$[\ell]$ to highlight the polynomial degree in time, the material derivative~$\Dt \Vht^{[\ell]} \not\subset \Vht^{[\ell-1]}$, unless~$\A = \Id$. 
    This fact introduces significant challenges in the stability and convergence analyses in the subsequent sections.
    \eremk
\end{remark}
Due to Lemma~\ref{lemma:div-preservation} and the hat notation introduced in Remark~\ref{rem:notation_conventions}, for all~$\vht \in \Vht$ and~$n = 1, \ldots, N$, it holds
\begin{equation}\label{eq:div-hat-identity}
    \divx \vht (\bx, t) = \frac{1}{\det J_t(\by)} \divy \hv_{h\tau}(\by, t) \qquad \forall (\bx, t) \in \Qn, \ \text{with~$\bx = \At(\by)$.}
\end{equation}
\subsection{Discrete kernel space}\label{sec:kernel-space}
For $n=1,2,\ldots,N$, we also define the discrete velocity kernel space
$$
    \Zht^{(n)} := \bigg\{ \vht \in \Vht^{(n)} \ :  \ (\divx \vht, \qht)_{\Qn} = 0  \ \ \forall \qht \in \Mht^{(n)}  \bigg\} \, .
$$
The global space $\Zht$ is then defined naturally, as in~\eqref{eq:globals:1} and~\eqref{eq:globals:2}. 
\begin{lemma}[Compatibility of the discrete spaces]\label{lem:compat}
    The discrete velocity kernel space is equivalent to
    $$
        \Zht^{(n)} = \bigg\{ \vht \in \Vht^{(n)} \ :  \  \divx \vht = 0 \ \ \textrm{in~$\Qn$}
    \bigg\} \, .
    $$
\end{lemma}
\begin{proof}
    Let~$n \in \{1, \ldots, N\}$ and~$\vht^{(n)} \in \Vht^{(n)}$. Recalling the hat and tilde notation introduced in Remark~\ref{rem:notation_conventions}, and then applying  identity~\eqref{eq:div-hat-identity} and a change of variables of the ensuing integral, for all~$\qht^{(n)} \in \Mht^{(n)}$, we have
    \begin{alignat}{3}
        \nonumber
        \int_{\In} \big(\divx \vht^{(n)}, \qht^{(n)}\big)_{\Omega(t)} \dt & = \int_{\In} \int_{\Omega(t)} \Big(\frac{1}{\det J_t} (\divy \hv_{h\tau}) \Big) \big(\At^{-1}(\bx) , t \big) \widetilde{q}_{h\tau} \big(\At^{-1}(\bx), t\big) \dx \dt \\
        \label{eq:div-t-0}
        & = \int_{\In} \int_{\Omega_0} \divy \hv_{h\tau} (\by, t) \widetilde{q}_{h\tau} (\by, t) \dy \dt. 
    \end{alignat}
    Assume now that $\vht^{(n)} \in \Zht^{(n)}$ satisfies~$\int_{\In} \big(\divx \vht, \qht^{(n)}\big)_{\Omega(t)} \dt=0$ for all $\qht^{(n)}\in\Mht^{(n)}$.
    By the well-known property of BDM elements in Lemma~\ref{lemma:commutativity}: $\divy (\Vhk \otimes \Pp{\ell}{\In}) = \Mhk \otimes \Pp{\ell}{\In}$.
    Therefore, we can find~$\widetilde{q}_{h\tau} \in \Mhk \otimes \Pp{\ell}{\In}$ such that~$\widetilde{q}_{h\tau}(\by, t) = \divy \hv_{h\tau}(\by, t)$. Choosing the corresponding~$\qht \in \Mht$ in identity~\eqref{eq:div-t-0} gives
    \begin{equation*}
        \Norm{\divy \hv_{h\tau}}{L^2(\In; L^2(\Omega_0))}^2 = 0,
    \end{equation*}
    which implies that~$\divy \hv_{h\tau} = 0$ in~$\Omega_0 \times \In$ and, together with 
    identity~\eqref{eq:div-hat-identity}, gives~$\divx \vht = 0$ in~$\Qn$. 
\end{proof}
\begin{remark}[Exact geometric representation]\label{rem:geo-exact}
    The proposed method represents exactly the geometry of the moving domain, which is allowed to be curved and may require high-order quadratures. 
    The same construction can be adapted to an isoparametric approach, i.e., approximating the map~$\At$ by a piecewise polynomial.
    \eremk
\end{remark}
\begin{remark}[Initial domain~$\Omega_0$]\label{rem:initial}
    We can obtain a simplicial partition~$\ThO$ of~$\Omega_0$ because $\Omega_0$ was assumed to be a polytopal domain.
    This simplification is introduced only for ease of exposition. A more general assumption could be that $\Omega_0$ is a Lipschitz domain described as the image of a map ${\mathcal G}: \widetilde{\Omega} \rightarrow \Omega_0$, for a polytopal parametric domain~$\widetilde{\Omega}$, the map ${\mathcal G}$ being piecewise regular with piecewise-regular inverse. In such a case, given a simplicial mesh on the parametric domain $\widetilde{\Omega}$, the discrete spaces $\Vhk$ and~$\Mhk$ defined in~$\Omega_0$ would be simply replaced by their (${\mathcal G}$-based) Piola-mapped counterparts, in the spirit of~\cite{Bertrand_Starke:2016}.
    \eremk
\end{remark}

\subsection{Space--time DG method}\label{sec:DG-method}
We define the following forms:
\begin{alignat}{3}
    \nonumber
    \aht(\uht, \vht) & := \sum_{n = 1}^{N} \aht^{(n)} (\uht, \vht) 
    := \sum_{n = 1}^{N} \int_{\tnmo}^{\tn} \ah^t(\uht, \vht) \dt  \\
    \nonumber
    & := \sum_{n = 1}^{N} \int_{\tnmo}^{\tn} \bigg[\sum_{\Kh \in \ThO} (\Nablax \uht, \Nablax \vht)_{\Kt} - \sum_{F \in \Fh} (\av{\Nablax \uht}_{\Ft} \nFt, \jump{\vht}_{\Ft})_{\Ft} \\
    \label{eq:aht}
    & \qquad \qquad \quad  - \sum_{F \in \Fh} (\jump{\uht}_{\Ft}, \av{\Nablax \vht}_{\Ft} \nFt)_{\Ft} + \sum_{F \in \Fh} (\sigma \hF^{-1} \jump{\uht}_{\Ft}, \jump{\vht}_{\Ft})_{\Ft}  \bigg] \dt, \\
    \nonumber
    \cht(\bw; \uht, \vht) & := \sum_{n = 1}^N \cht^{(n)} (\bw; \uht, \vht) 
    := \sum_{n = 1}^{N} \int_{\tnmo}^{\tn} c_h^t(\bw;\uht, \vht) \dt  \\
    \label{eq:cht}
    & := \sum_{n = 1}^N \int_{\In} \bigg[\sum_{K \in \ThO} \big((\Nablax  \uht) \bw, \vht \big)_{\Kt} - \sum_{F \in \Fho} \big((\bw \cdot \nFt) \jump{\uht}_{\Ft}, \av{\vht}_{\Ft} \big)_{\Ft} \bigg] \dt,
\end{alignat}
where the stability parameter~$\sigma > 0$ will be chosen ``large enough", as usual.

Furthermore, for all~$t \in [0, T]$, we define
$$
    (v,w)_{\Omega_h(t)} := \sum_{K \in \ThO} (v,w)_{\Kt} \qquad \forall v,w: \Omega(t) \rightarrow {\mathbb R} \, , \textrm{ piecewise regular} ,
$$
with the obvious analog for~$d$-vector-valued functions.

The proposed space--time DG formulation of~\eqref{eq:model-problem-ALE} reads: find~$(\uht, \pht) \in \Vht \times \Mht$ such that
\begin{subequations}\label{method}
    \begin{alignat}{3}
        \nonumber 
        \sum_{n = 1}^{N} \int_{\In} (\Dt \uht, \vht)_{\Omega_h(t)} \dt  + \sum_{n = 1}^{N - 1} \big(\jump{\uht}_n, & \vht(\cdot, \tn^+) \big)_{\Omega(\tn)}  + (\uht(\cdot,0), \vht(\cdot, 0))_{\Omega_0} - \cht(\bw; \uht, \vht)\\
        + \nu \aht(\uht, \vht) + (\pht, \divx \vht)_{\QT} & = (\f, \vht)_{\QT} + (\u_0, \vht(\cdot, 0))_{\Omega_0} \qquad  \forall \vht \in \Vht, \\
        (\divx \uht, \qht)_{\QT} & = 0 \hspace{5.1cm} \forall \qht \in \Mht .
    \end{alignat}
\end{subequations}
By Lemma~\ref{lem:compat}, it is simple to check that the discrete problem~\eqref{method} can be written equivalently as:
find~$\uht \in \Zht$ such that
\begin{alignat}{3}
    \nonumber
    \Bht(\uht, \zht) & := \sum_{n = 1}^{N} \int_{\In} (\Dt \uht, \zht)_{\Omega_h(t)} \dt  + \sum_{n = 1}^{N - 1} \big(\jump{\uht}_n, \zht(\cdot, \tn^+) \big)_{\Omega(\tn)}  + (\uht(\cdot,0), \zht(\cdot, 0))_{\Omega_0} \\
    \label{method:Z}
    & \quad - \cht(\bw; \uht, \zht)
    + \nu \aht(\uht, \zht) = (\f, \zht)_{\QT} + (\u_0, \zht(\cdot, 0))_{\Omega_0} \quad  \forall \zht \in \Zht \, . 
\end{alignat}
\section{Stability analysis}\label{sec:4}
This section is devoted to establishing the well-posedness of the space--time method~\eqref{method}.

Henceforth, we will use~$a \lesssim b$ to denote the existence of a positive constant~$C$ independent of the mesh size~$h$, the time steps~$\tau_n$, and the viscosity coefficient~$\nu$ such that~$a \le C b$. In particular, we trace explicitly the dependence of all constants on~$\nu$. We will also use~$a \simeq b$ to indicate that~$a \lesssim b$ and~$b \lesssim a$.

For all~$t \in [0, T]$ and all piecewise-regular (with respect to the mapped mesh) functions on $\Omega(t)$, we define
\begin{equation*}
    \Norm{\vh}{1,h,t}^2 := \sum_{\Kh \in \ThO} \Norm{\nablax \vh}{\L^2(\Kt)}^2 + \sum_{F \in \Fh} \hF^{-1} \Norm{\jump{\vh}_{\Ft}}{\L^2(\Ft)}^2.
\end{equation*}
In the next lemma, we collect several bounds related to the Piola transformation $\phi_t$; the proof and the specific constants involved are provided in Appendix~\ref{app:bounds}.
\begin{lemma}\label{lemma:equivalence}
    Let Assumption~\ref{asm:At} on the map~$\A$ hold.
    Let $\hv \in \boldsymbol {H}^2(\ThO)$  and $\widetilde{q} \in L^2(\Omega_0)$.
    For any $K \in \ThO$, any $F \in \Fh$, and any $t \in [0,T]$, the following inequalities hold:
    \begin{subequations}
        \begin{alignat}{3}
            \label{eq:trace.NEW}
            \Norm{\hv}{\L^2(\partial K)}^2 &\simeq \Norm{\phi_t\hv}{\L^2(\partial \Kt)}^2, \\
            \label{eq:trace-grad.NEW}
            \Norm{\Nablax (\phi_t \hv)}{\L^2(\partial \Kt)}^2  & \lesssim \hK^{-1} \Norm{\widehat{\v}}{\L^2(K)}^2 + (\hK + \hK^{-1})  \Norm{\Nablay \widehat{\v}}{\L^2(K)}^2 + \hK \Norm{\Nablay^2 \widehat{\v}}{\L^2(K)}^2, \\ 
            \label{eq:equiv-vh.1}
            \Norm{\Nablax (\phi_t \hv)}{\L^2(\Kt)}^2 & \lesssim \Norm{\hv}{\L^2(K)}^2 + \Norm{\Nablay \hv}{\L^2(K)}^2, \\ 
            \label{eq:equiv-vh.2}
            \Norm{\Nablay \hv }{\L^2(K)}^2 & \lesssim \Norm{\phi_t \hv}{\L^2(\Kt)}^2 + \Norm{\Nablax (\phi_t \hv)}{\L^2(\Kt)}^2, \\ 
            \label{eq:equiv-vh-2}
            \Norm{\jump{\hv}_{F}}{\L^2(F)}^2 &\simeq \Norm{\jump{\phi_t \hv}_{\Ft}}{\L^2(\Ft)}^2, \\
            \label{eq:equiv-vh-3}
            \Norm{\hv}{\L^2(\Omega_0)}^2 &\simeq \Norm{\phi_t \hv}{\L^2(\Omega(t))}^2, \\
            \label{eq:equiv-qht}
            \Norm{\widetilde{q} \circ \At^{-1} }{L^2(\Omega(t))}^2  & \simeq \Norm{\widetilde{q}}{L^2(\Omega_0)}^2.
        \end{alignat}
    \end{subequations}
    Moreover, if $\hv \in \Vhk$, then the following discrete trace estimates hold:
    \begin{subequations}
        \begin{alignat}{3}
            \label{eq:trace.NEW.NEW}
            \Norm{\phi_t\hv}{\L^2(\partial \Kt)}^2   &\lesssim \hK^{-1}   \Norm{\phi_t\hv}{\L^2(\Kt)}^2, \\
            \label{eq:trace-grad.NEW.NEW}
            \Norm{\Nablax (\phi_t\hv)}{\L^2(\partial \Kt)}^2  & \lesssim \hK^{-1}  \Norm{\Nablax (\phi_t\hv)}{\L^2(\Kt)}^2.
        \end{alignat}
    \end{subequations}
    All the hidden constants above depend only on $d$, $k$, the mesh regularity parameter of $\ThO$, and the regularity of the map~$\At$.
\end{lemma}
The next lemmas provide the ingredients to prove the well-posedness of the space--time method~\eqref{method} in Theorem~\ref{thm:well-pos} below.
\begin{lemma}[Discrete Poincar\'e inequality]\label{lem:dpi}
    Let Assumption~\ref{asm:At} on the map~$\A$ hold. For any $t \in [0,T]$ and any mesh $\ThO$ in the family, the following bound holds:
    $$
        \Norm{\v}{\L^2(\Omega(t))} \lesssim \Norm{\v}{1,h,t} \qquad \forall \v \in \H^1(\Tht) \, ,
    $$
    with the involved constant independent of~$h$ and~$t$.
\end{lemma}
\begin{proof}
    Let $\v \in \H^1(\Tht) $.
    Recalling Assumption \ref{asm:At} and Remark \ref{rem:assumptions-At}, a mapping argument easily gives
    $$
        \Norm{\v}{\L^2(\Omega(t))} \lesssim \Norm{\v\circ\At}{\L^2(\Omega_0)} \, .
    $$
    Moreover, it can be immediately verified that~$\v\circ\At \in \H^1(\ThO)$. 
    Therefore, we first apply the Poincar\'e--Friedrich inequality on~$\ThO$ from~\cite{Brenner:2003},
    and subsequently use again a mapping argument (treating jump and gradient terms separately) to obtain
    $$
        \Norm{\v\circ\At}{\L^2(\Omega_0)} \lesssim \Norm{\v\circ\At}{1,h,0} \lesssim \Norm{\v}{1,h,t} \, .
    $$
    The proof is concluded combining the two bounds above.
\end{proof}
\begin{lemma}[Coercivity of~$\aht$]\label{lemma:coercivity-aht}
Under Assumption~\ref{asm:At} on the map~$\A$, for~$\sigma$ sufficiently large, there exists a positive constant $C_a$ such that, for~$n = 1, \ldots, N$, it holds
\begin{equation}\label{DG-coerc}
    \aht^{(n)}(\vht, \vht) \ge C_a \int_{\In} \| \vht\|_{1,h,t}^2 \dt  \qquad \forall \vht \in \Vht.
\end{equation}
\end{lemma}
\begin{proof}
    The result follows from standard DG arguments and the trace inequality~\eqref{eq:trace-grad.NEW.NEW}.
\end{proof}
The previous lemma allows us to prove the following discrete inf--sup condition. To simplify the proof, we set~$\{\psi_j^{(n)}\}_{j = 0}^{\ell}$ as the Legendre basis on the interval~$\In$, which satisfies the following orthogonality property:
\begin{equation}\label{eq:orthogonality-Leg}
    \int_{\In} \psi_i^{(n)}  \psi_j^{(n)} \dt  = \delta_{ij} \frac{\tau_n}{2i+1} ,
\end{equation}
where~$\delta_{ij}$ is the standard Kronecker delta function.
\begin{proposition}[Inf--sup condition]\label{prop:inf-sup}
    There exists~$C_\beta >0$ independent of $h$ and~$\tau$ such that, for $n=1,2,\ldots ,N$, it holds
    \begin{equation*}
        \sup_{\vht^{(n)}\in\Vht^{(n)}} \frac{\displaystyle  \int_{\In} \big(\divx \vht^{(n)}(\cdot,t), \qht^{(n)}(\cdot,t)\big)_{\Omega(t)} \dt}{\displaystyle \int_{I_n} \| \vht^{(n)}(\cdot,t)\|_{1,h,t}^2 \dt}
        \ge C_{\beta} \left( \int_{I_n} \Norm{\qht^{(n)}(\cdot,t)}{L^2(\Omega(t))}^2 \dt \right)^{1/2} \quad \forall \qht \in \Mht .
    \end{equation*}
\end{proposition}
\begin{proof}
    Let $n \in \{1,2,\ldots , N\}$, and~$\qht^{(n)} \in \Mht^{(n)}$, with~$\widetilde{q}_{h\tau}^{(n)}(\by,t) = \sum_{j = 0}^{\ell} \widetilde{\beta}_{h,j} (\by) \psi_j^{(n)}(t)$; see~\eqref{def:Mht}.
    By the inf--sup stability of standard BDM elements endowed with DG-norms (see~\cite[Thm.~2.2]{Lederer_Schrobel:2018}), for each $\widetilde{\beta}_{h,j} \in \Mhk$, we can find $\ba_{h,j} \in \Vhk$ such that $\divy\ba_{h,j}=\widetilde{\beta}_{h,j}$ and
    \begin{equation}\label{eq:bdm:F2}
        \| \ba_{h,j} \|_{1,h,0} \lesssim \Norm{\widetilde{\beta}_{h,j}}{L^2(\Omega_0)} \, .
    \end{equation}
    We now take such~$\{ \ba_{h,j} \}_{j=0}^{\ell}$ and consider the choice $\vht^{(n)}(\bx,t) := \sum_{i = 0}^{\ell} (\phi_t \ba_{h, i})(\bx)\psi_i^{(n)}(t)$. It is evident that~$\divy \hv^{(n)}_{h\tau} = \widetilde{q}^{(n)}_{h\tau}$ in~$\Omega_0 \times \In$. Therefore, recalling~\eqref{eq:div-t-0}, and using bound~\eqref{eq:equiv-qht} from Lemma~\ref{lemma:equivalence}, we obtain
    \begin{alignat}{3}
        \int_{\In} \big(\divx \vht^{(n)}, \qht^{(n)}\big)_{\Omega(t)} \dt
        \label{eq:temp:4}
        & = \Norm{\widetilde{q}_{h\tau}^{(n)}}{L^2(\In; L^2(\Omega_0))}^2 \gtrsim \int_{\In} \Norm{\qht^{(n)}(\cdot, t)}{L^2(\Omega(t))}^2 \dt. 
    \end{alignat}
    Furthermore, by the triangle inequality, we have
    \begin{equation}\label{eq:threelions}
        \int_{I_n} \| \vht^{(n)}(\cdot,t)\|_{1,h,t}^2 \dt = \int_{I_n} \| \sum_{j = 0}^{\ell} (\phi_t \ba_{h, j}) (\bx) \psi_j^{(n)}(t) \|_{1,h,t}^2 \dt
        \le  \sum_{j = 0}^{\ell} \int_{I_n}  \| (\phi_t \ba_{h, j}) \|_{1,h,t}^2  |\psi_j^{(n)}(t)|^2 \dt \, .
    \end{equation}
    We preliminarily note that, due to the orthogonality property~\eqref{eq:orthogonality-Leg}, it holds
    \begin{equation}\label{eq:norm:equiv}
        \Norm{\widetilde{q}_{h\tau}^{(n)}}{L^2(\In; L^2(\Omega_0))}^2 = \sum_{j = 0}^{\ell} \frac{\tau_n}{2j+1} \Norm{\hb_{h,j}}{L^2(\Omega_0)}^2 .
    \end{equation}
    Starting from \eqref{eq:threelions}, we apply bounds~\eqref{eq:equiv-vh.2} and~\eqref{eq:equiv-vh-2}, followed by the Poincar\'e--Friedrich inequality on~$\ThO$ from~\cite{Brenner:2003}, and finally use estimate~\eqref{eq:bdm:F2} and the orthogonality property~\eqref{eq:orthogonality-Leg} to obtain
    \begin{equation*}
        \begin{aligned}
            \int_{I_n} \| \vht^{(n)}(\cdot,t)\|_{1,h,t}^2 \dt & \lesssim
            \sum_{j = 0}^{\ell} \big( \| \ba_{h, j} \, \|_{1,h,0}^2 + \| \ba_{h, j} \|_{\L^2(\Omega_0)}^2 \big) \int_{I_n} |\psi_j^{(n)}|^2 \dt
            \\
            & \lesssim \sum_{j = 0}^{\ell} \| \ba_{h, j} \, %
            \|_{1,h,0}^2  \int_{I_n} |\psi_j^{(n)}|^2 \dt
            \lesssim \sum_{j = 0}^{\ell} \frac{\tau_n}{2j+1} \Norm{\widetilde{\beta}_{h,j}}{L^2(\Omega_0)}^2 . 
        \end{aligned}
    \end{equation*} 
    The above bound, combined with identity~\eqref{eq:norm:equiv} and bound~\eqref{eq:equiv-qht}, gives
    \begin{equation}\label{eq:temp:6}
        \int_{I_n} \| \vht^{(n)}(\cdot,t)\|_{1,h,t}^2 \dt \lesssim \Norm{\widetilde{q}_{h\tau}^{(n)}}{L^2(\In; L^2(\Omega_0))}^2
        \lesssim \int_{I_n}  \Norm{\qht^{(n)}(\cdot,t)}{L^2(\Omega(t))}^2 \dt .
    \end{equation}
    The result now follows by standard theory for the inf--sup condition, combining~\eqref{eq:temp:4} and~\eqref{eq:temp:6}. 
\end{proof}
The stability analysis also relies on the following local Reynolds identities, which we write on generic mapped simplices; see~\cite[Lemma 2.2]{Bonito_Kyza_Nochetto:2013}.
\begin{lemma}[Reynolds identities for regular functions]
    For all~$K \in \ThO$, $n \in \{1, \ldots, N\}$, and~$t \in \In$, it holds
    \begin{equation}\label{eq:cont:div}
        \ddt \int_{\Kt} \bvphi \cdot \v \dx 
        = \int_{\Kt} \v \cdot (\Dt \bvphi + \bvphi (\divx  \bw)) \dx + \int_{\Kt} \bvphi \cdot \Dt \v \dx \qquad \forall \v, \bvphi \in \H^1(\calK_n).
    \end{equation}
\end{lemma}
As a consequence, for~$n = 1, \ldots, N$, the following identity holds for all $\v \in \H^1(\calK_n)$:
$$
\frac12 \Norm{\v(\cdot, \tn^-)}{\L^2(K_{t_n})}^2 - \frac12 \Norm{\v(\cdot, \tnmo^+)}{\L^2(K_{t_{n-1}})}^2  
= \int_{\In} \int_{\Kt} \Big[ \Dt \v \cdot \v + \frac12 (\divx \bw) |\v|^2 \Big] \dx \dt \ .
$$
The previous result yields the following crucial lemma, valid for discrete functions.
\begin{lemma}[Discrete Reynolds identities]
Let $\vht,\vvht \in \Vht$. For all~$n \in \{1, \ldots, N\}$ and~$t \in \In$, it holds
\begin{subequations}
    \begin{equation}\label{eq:first}
        \begin{aligned}
            \ddt \int_{\Omega(t)} \vvht \cdot \vht \dx
            & = \int_{\Omega_h(t)}\big( \vvht \cdot \Dt \vht + \vht \cdot \Dt \vvht\big) \dx - c_h^t(\bw; \vvht, \vht) - c_h^t(\bw; \vht, \vvht),
        \end{aligned}
    \end{equation}
    and, for each~$n=1,2,\ldots,N$,
    \begin{equation}\label{eq:second}
        \begin{aligned}
            & \frac12 \Norm{\vht(\cdot, \tn^-)}{\L^2(\Omega(t_n))}^2  - \frac12 \Norm{\vht(\cdot, \tnmo^+)}{\L^2(\Omega(t_{n-1}))}^2  = \int_{\In} \int_{\Omega_h(t)} \Dt \vht \cdot \vht  \dx  \dt  -  \cht^{(n)}(\bw; \vht, \vht).
        \end{aligned}
    \end{equation}
\end{subequations}
\end{lemma}
\begin{proof}
    For each~$n \in \{1, \ldots, N\}$ and~$t \in \In$, integration by parts and standard DG arguments yield
    \begin{equation}\label{div-DG}
        \int_{\Omega(t)} (\divx  \bw) \vvht\cdot\vht  \dx 
        = - c_h^t(\bw; \vvht, \vht) - c_h^t(\bw; \vht, \vvht) \, .
    \end{equation}
    Writing the integral in~$\Omega(t)$ as a sum over all elements~$K \in \ThO$, then applying~\eqref{eq:cont:div} elementwise, and finally using~\eqref{div-DG}, we obtain
    $$
        \begin{aligned}
            \ddt \int_{\Omega(t)} \vvht \cdot \vht \dx
            & = \int_{\Omega_h(t)} \big(\vvht \cdot \Dt \vht
            + \vht \cdot \Dt \vvht\big) \dx
            - c_h^t(\bw; \vvht, \vht)  - c_h^t(\bw; \vht, \vvht).
        \end{aligned}
    $$
    The second identity~\eqref{eq:second} is an immediate consequence of~\eqref{eq:first} with $\vvht=\vht$, as integration in time over~$\In$ leads to
    \begin{alignat*}{3}
        \frac12 \Norm{\vht(\cdot, \tn^-)}{\L^2(\Omega(t_n))}^2 - \frac12 \Norm{\vht(\cdot, \tnmo^+)}{\L^2(\Omega(t_{n-1}))}^2  
        & = \frac12 \int_{\In} \Big[ \ddt  \int_{\Omega(t)} |\vht|^2 \dx \Big]\dt  \\
        & = \int_{\In} \Big[ \int_{\Omega_h(t)} \Dt \vht \cdot \vht  \dx
        -  c_h^t(\bw; \vht, \vht) 
        \Big] \dt \, .
    \end{alignat*}
\end{proof}
We are now in a position to prove the main result in this section. The following result shows the existence and uniqueness of a solution to problem \eqref{method} and its stability at the time nodes $\{ t_n \}_{n=0}^N$. 
\begin{theorem}[Well-posedness]\label{thm:well-pos}
    Under the assumptions of Lemma~\ref{lemma:coercivity-aht}, there exists a unique solution~$(\uht, \pht) \in \Vht \times \Mht$ to problem~\eqref{method}. Furthermore, such a solution satisfies the following nodal stability bound  for~$m \in \{ 1, 2, \ldots, N \}$:
    \begin{alignat}{3}
        \nonumber
            \frac{1}{2}\Norm{\uht(\cdot, t_m^-)}{\L^2(\Omega(t_m))}^2 & + \frac{1}{2}\sum_{n = 1}^{m-1} \Norm{\jump{\uht}_n}{\L^2(\Omega(t_n))}^2 + \frac{1}{4}\Norm{\uht(\cdot, 0)}{\L^2(\Omega_0)}^2  + \frac{1}{2} \sum_{n = 1}^{m}  \int_{\In} %
            \nu \|\uht (\cdot,t)\|_{1,h,t}^2
            \dt \\ 
            \label{eq:node.stab.bound}
            &\quad \lesssim 
            \frac{1}{\nu} \sum_{n = 1}^{m} \int_{\In} \Norm{\f (\cdot,t)}{\L^2(\Omega(t))}^2 \dt +  \Norm{\u_0}{\L^2(\Omega_0)}^2.
        \end{alignat}
\end{theorem}
\begin{proof}
    The discrete linear problem \eqref{method} can be solved sequentially for $n=1,2,\ldots,N$. For each~$n$, the problem has the classical mixed form: find $(\uht^{(n)}, \pht^{(n)}) \in \Vht^{(n)} \times \Mht^{(n)}$ such that
    $$
        \left\{
        \begin{aligned}
            & {\mathbb A}_h^{(n)}(\uht^{(n)},\vht) + {\mathbb B}_h^{(n)}(\vht,\pht^{(n)}) =  (\f, \vht)_{\Qn} 
            + \big(\uht^{(n-1)}(\cdot,t_{n-1}^-), \vht(\cdot,t_{n-1}^+)\big)_{\Omega_{n-1}}
            & & \quad  \forall \vht \in \Vht^{(n)}, \\
            & {\mathbb B}_h^{(n)}(\uht^{(n)},\qht) = 0 & & \quad \forall \qht \in \Mht^{(n)},
        \end{aligned}
        \right.
    $$
    with the obvious right-hand side modification for $n=1$, and where  
    $$
        \begin{aligned}
            {\mathbb A}_h^{(n)}(\uht,\vht) := &
            \int_{\In} (\Dt \uht, \vht)_{\Omega_h(t)} \dt  
            + \big(\uht(\cdot, \tnmo^+),  \vht(\cdot, \tnmo^+) \big)_{\Omega(\tnmo)}  \\
            & - \cht^{(n)}(\bw; \uht, \vht)
            + \nu \aht^{(n)}(\uht, \vht)  \, , \\
            {\mathbb B}_h^{(n)}(\uht,\qht) := &  (\qht, \divx \uht)_{\Qn} \, . 
        \end{aligned}
    $$
    Using the discrete Reynolds identity~\eqref{eq:second} and the coercivity in~\eqref{DG-coerc} of the bilinear form~$\aht(\cdot, \cdot)$, we obtain 
    \begin{alignat*}{3}
        {\mathbb A}_h^{(n)}(\uht^{(n)},\uht^{(n)}) & = 
        \frac12 \Norm{\uht^{(n)}(\cdot, \tn^-)}{\L^2(\Omega(t_n))}^2 - \frac12 \Norm{\uht^{(n)}(\cdot, \tnmo^+)}{\L^2(\Omega(t_{n-1}))}^2 \\
        & \quad 
        + \big(\uht^{(n-1)}(\cdot,t_{n-1}^+), \uht^{(n-1)}(\cdot,t_{n-1}^+)\big)_{\Omega_{n-1}}  + \nu \int_{\In} \aht\big(\uht^{(n)},\uht^{(n)}\big) \dt \\
        &
        \geq \frac12 \Norm{\uht^{(n)}(\cdot, \tn^-)}{\L^2(\Omega(t_n))}^2 + 
        \frac12 \Norm{\uht^{(n)}(\cdot, \tnmo^+)}{\L^2(\Omega(t_n))}^2  + \nu C_a \int_{\In} %
         |\uht^{(n)}|_{1,h}^2
        \dt ,
    \end{alignat*}
    so that the bilinear form~${\mathbb A}_h^{(n)}$ is coercive on~$\Vht^{(n)}$. Moreover, the inf--sup stability of the spaces~$\Vht^{(n)} \times \Mht^{(n)}$ with respect to the form~${\mathbb B}_h^{(n)}$ has been shown in Proposition~\ref{prop:inf-sup}.
    As a consequence of theset two properties, standard theory of mixed methods \cite[\S4.2]{Boffi_Brezzi_Fortin:2013} yields existence and uniqueness of a solution. 
    
    Let $\uht \in \Vht$ be the solution to~\eqref{method:Z}, and let~$m \in \{ 1, \ldots, N \}$ be fixed. We define $\vht \in \Vht$ as
    \begin{equation*}
        \vht {}|_{\Qj} := 
        \begin{cases}
            \uht & \text{ if~$j  \in \{1, \ldots, m\}$}, \\
            \bO & \text{ otherwise}.
        \end{cases}
    \end{equation*}
    Taking this $\vht$ as a test function in~\eqref{method:Z}, we get
    \begin{equation}\label{eq:node.stab.bound.initial}
        \begin{aligned}
            & \sum_{n = 1}^{m} \int_{\In} (\Dt \uht, \uht)_{\Omega_h(t)} \dt  + \sum_{n = 1}^{m - 1} \big(\jump{\uht}_n, \uht(\cdot, \tn^+) \big)_{\Omega(\tn)}  + \Norm{\uht(\cdot,0)}{\L^2(\Omega_0)}^2\\
            & \quad - \sum_{n = 1}^{m} \cht^{(n)}(\bw; \uht, \uht)
            + \nu \sum_{n = 1}^{m} \aht^{(n)}(\uht, \uht) = \sum_{n = 1}^{m} \int_{\In} (\f, \uht)_{\Omega_h(t)} \dt + (\u_0, \uht(\cdot, 0))_{\Omega_0}. 
        \end{aligned}
    \end{equation}
    We denote by $\mathrm{LHS}$ and~$\mathrm{RHS}$ the left- and right-hand sides of the above equation, respectively.
    For the first and fourth terms in LHS, using the discrete Reynolds identity~\eqref{eq:second} and recalling the definition of $\cht(\cdot; \cdot, \cdot)$, we obtain
    \begin{equation}\label{eq:node.stab.bound.lhs.term1}
        \begin{aligned}
            \sum_{n = 1}^{m} \int_{\In} (\Dt \uht, \uht)_{\Omega_h(t)} \dt & - \cht(\bw; \uht, \uht) \\
            &= \frac{1}{2}\sum_{n = 1}^{m} 
             \big( \Norm{\uht(\cdot, \tn^-)}{\L^2(\Omega(t_n))}^2 - \Norm{\uht(\cdot, \tnmo^+)}{\L^2(\Omega(\tnmo))}^2 
            \big).
         \end{aligned}
    \end{equation}
    For the second term in~$\mathrm{LHS}$, recalling the definition in~\eqref{eq:time-jump} of the time-jump operator $\jump{\cdot}_n$, and using the relation $(\boldsymbol{a}- \boldsymbol{b})\cdot \boldsymbol{a} =\frac{1}{2} (|\boldsymbol{a}|^2 + |\boldsymbol{a}-\boldsymbol{b}|^2 - |\boldsymbol{b}|^2)$ with $\boldsymbol{a}=\uht (\cdot,\tn^+)$ and $\boldsymbol{b}=\uht (\cdot,\tn^-)$, we have 
    \begin{equation}\label{eq:node.stab.bound.lhs.term2}
        \begin{aligned}
            \sum_{n = 1}^{m - 1} \big(\jump{\uht}_n, \uht(\cdot, \tn^+) \big)_{\Omega(\tn)} 
            &= \sum_{n = 1}^{m-1} \frac{1}{2} \left[ 
            \Norm{\uht(\cdot, \tn^+)}{\L^2(\Omega(t_n))}^2 + \Norm{\jump{\uht}_n}{\L^2(\Omega(t_n))}^2 - \Norm{\uht(\cdot, \tn^-)}{\L^2(\Omega(\tn))}^2
            \right].
        \end{aligned}
    \end{equation}
    As for the last term in $\mathrm{LHS}$, we use the coercivity in~\eqref{DG-coerc} of~$\aht(\cdot, \cdot)$ to obtain
    \begin{equation}\label{eq:node.stab.bound.lhs.term3}
        \nu \aht(\uht, \uht) \geq \sum_{n = 1}^{m} \nu C_a \int_{\In} \|\uht (\cdot,t)\|_{1,h,t}^2 \dt.
    \end{equation}
    Combining~\eqref{eq:node.stab.bound.lhs.term1}, \eqref{eq:node.stab.bound.lhs.term2}, and~\eqref{eq:node.stab.bound.lhs.term3} with identity~\eqref{eq:node.stab.bound.initial} yields
    \begin{equation}\label{eq:node.stab.bound.lhs.bound}
        \begin{aligned}
            \mathrm{LHS} &\geq \frac{1}{2}\bigg[\Norm{\uht(\cdot, t_m^-)}{\L^2(\Omega(t_m))}^2 +\sum_{n = 1}^{m-1} \Norm{\jump{\uht}_n}{\L^2(\Omega(t_n))}^2 + \Norm{\uht(\cdot, 0)}{\L^2(\Omega_0)}^2\bigg] \\ 
            &\quad + \sum_{n = 1}^{m} \int_{\In} %
            \nu C_a \|\uht (\cdot,t)\|_{1,h,t}^2 %
            \dt.
        \end{aligned}
    \end{equation}
    Applying the Cauchy--Schwarz and the Young inequalities to the~$\mathrm{RHS}$ of \eqref{eq:node.stab.bound.initial}, along with Lemma~\ref{lem:dpi}, we obtain
    \begin{equation}\label{eq:stab.bound.rhs.upper.bound}
        \begin{split}
            \mathrm{RHS} & \leq \frac{1}{2 \varepsilon \nu} \sum_{n = 1}^{m} \int_{\In} \Norm{\f (\cdot,t)}{\L^2(\Omega(t))}^2 \dt +  \varepsilon \frac{\nu \CPF}{2} \sum_{n = 1}^{m} \int_{\In} \|\uht (\cdot,t)\|_{1,h,t}^2 \dt \\
            & \quad + \Norm{\u_0}{\L^2(\Omega_0)}^2 + \frac{1}{4}\Norm{\uht(\cdot, 0)}{\L^2(\Omega_0)}^2,
        \end{split}
    \end{equation}
    where we denoted by~$\CPF$ the constant in Lemma \ref{lem:dpi}.
    From~\eqref{eq:node.stab.bound.lhs.bound} and~\eqref{eq:stab.bound.rhs.upper.bound}, straightforward algebraic manipulations and choosing $\varepsilon = C_a/\CPF > 0$ easily give
    \begin{alignat}{3}
        \nonumber
        &\frac{1}{2}\Norm{\uht(\cdot, t_m^-)}{\L^2(\Omega(t_m))}^2 + \frac{1}{2}\sum_{n = 1}^{m-1} \Norm{\jump{\uht}_n}{\L^2(\Omega(t_n))}^2 + \frac{1}{4}\Norm{\uht(\cdot, 0)}{\L^2(\Omega_0)}^2 + \frac{1}{2} \sum_{n = 1}^{m}  \int_{\In} \nu \|\uht (\cdot,t)\|_{1,h,t}^2 
        \dt\\ 
        \label{eq:aux-LHS-RHS}
        &\quad \leq
        C \bigg[ 
        \frac{1}{\nu} \sum_{n = 1}^{m} \int_{\In} \Norm{\f (\cdot,t)}{\L^2(\Omega(t))}^2 \dt + \Norm{\u_0}{\L^2(\Omega_0)}^2\bigg],
    \end{alignat}
    with $C \coloneqq \max\{\CPF/(2 C_a), \, 1 \} / \min \{ 1,\, C_a \}$. 
\end{proof}
\begin{remark}[Pressure-robust stability bound]
    Thanks to the property in Lemma \ref{lem:compat}, which implies that the discrete kernel space~$\Zht$ is contained in the continuous one, the forcing~$\f$ on the right-hand side of bound~\eqref{eq:node.stab.bound} could be substituted by its Helmholtz--Hodge projection, underlining the pressure-robustness of the scheme. We refer to~\cite{john2017divergence} for more details and an extended discussion on the relevance of such a property.
    \eremk
\end{remark}
\section{Convergence analysis}\label{sec:5}
We now focus on the derivation of a priori estimates for the error in the energy norm.
\subsection{Space--time projection}
We recall the definition and properties of some auxiliary operators. We start with the Thom\'ee projection in time~$\Pt$ (see~\cite[eq.~(12.9) in Ch.~12]{Thomee:2006}), which is a classical tool in the analysis of DG time discretizations. 
\begin{definition}[The Thom\'ee projection~$\Pt$] \label{def:Thomee-projection}
    For~$\ell \in \N$ and a Hilbert space~$\calH$ with inner product~$(\cdot, \cdot)_{\calH}$, the projection operator~$\Pt: H^1(0, T; \calH) \to \calH \otimes \Pp{\ell}{\Tt}$ is defined for any~$v \in H^1(0, T; \calH)$ as follows: for~$n = 1, \ldots, N$,
    \begin{subequations}
        \begin{alignat*}{3}
            \Pt v(\tn^-) & = v(\tn), \\
            \int_{\In} (\Pt v, w_{\ell-1})_{\calH} \dt & = \int_{\In} (v, w_{\ell-1})_{\calH} \dt & & \quad \forall w_{\ell - 1} \in \calH \otimes \Pp{\ell-1}{\Tt}. 
        \end{alignat*}
    \end{subequations}
\end{definition}
The stability of~$\Pt$ in the~$L^{\infty}(0, T; \calH)$ norm yields the following approximation estimates (see~\cite[\S69.3.2]{Ern_Guermond:2021} and~\cite[Lemma~4.2]{Beirao-Gomez-Dassi:2025}).
\begin{lemma}[Estimates for~$\Pt$]\label{lemma:estimates-Pt}
    Let~$r \in [1, \infty]$, and~$\calH$ a Hilbert space with inner product~$(\cdot, \cdot)_{\calH}$. Then, for~$n = 1, \ldots, N$, the following estimate holds:
    \begin{equation*}
        \Norm{\Pt v - v}{L^r(\In; \calH)} \lesssim \tau_n^s\Norm{\dpt^{(s)} v}{L^r(\In; \calH)} \qquad \forall v \in W^{s,r}(\In; \calH), \ 1 \le s \le \ell+1.
    \end{equation*}
\end{lemma}
Moreover, we denote by~$\IBDM : \H^1(\Omega_0) \to \Vhk$ the standard BDM interpolant (see~\cite[\S2.5.2]{Boffi_Brezzi_Fortin:2013}) in the reference domain~$\Omega_0$, which, with a slight abuse of notation, we assume to act on the spatial variable in the space--time setting. Denoting by~$\PihM : L^1(\Omega_0) \to \Mhk$, the standard $L^2(\Omega_0)$-orthogonal projection, we recall the well-known commutativity property of~$\IBDM$ (see, e.g., \cite[\S2.5.6]{Boffi_Brezzi_Fortin:2013}).
\begin{lemma}[Commutativity of~$\IBDM$]\label{lemma:commutativity}
    The following identity holds:
    \begin{equation*}
        \divy \IBDM \hv = \PihM \divy \hv \qquad \forall \hv \in \H^1(\Omega_0).
    \end{equation*}
\end{lemma}
The next lemma concerns the approximation properties of the interpolant~$\IBDM$ (see, e.g., \cite[Prop. 2.5.1]{Boffi_Brezzi_Fortin:2013}).
\begin{lemma}[Estimates for~$\IBDM$]\label{lemma:estimates-IBDM}
    For all~$K \in \ThO$, $1 \le m \le k+1$, and~$\v \in \H^m(K)$, it holds:
    \begin{equation*}
        \Norm{\IBDM \v - \v}{\L^2(K)} + \hK\Norm{\Nablay (\IBDM \v - \v) }{\L^2(K)} \lesssim \hK^{m} \semiNorm{\v}{\H^{m}(K)}.
    \end{equation*}
\end{lemma}
Finally, we define the space--time projection~$\Piht$ in the discrete space~$\Vht$.
\begin{definition}[Space--time interpolation operator $\Piht$]\label{def:space-time-interpolant}
    For any function~$\v \in H^1(\H^1; \QT)$, we define its space--time interpolation~$\Piht \v \in \Vht$ as~$\Piht \v = \phi \big((\Pt \circ \IBDM)\hv \big)$.
\end{definition}
Due to Lemma~\ref{lemma:div-preservation} and the commutativity property in Lemma~\ref{lemma:commutativity}, for all~$\v \in H^1(\H^1; \QT)$ and for a.e. $(\bx, t) = (\At(\by), t) \in \QT$, it holds
\begin{equation*}
    \begin{split}
        \divx \Piht \v(\bx, t) 
        = \frac{1}{\det J_t(\by)} \Big[ \divy \big( (\Pt \circ \IBDM) \hv \big) \Big] (\by, t) 
        = \frac{1}{\det J_t(\by)} \Big[ \Pt \big( \PihM (\divy \hv) \big) \Big] (\by, t).
    \end{split}
\end{equation*}
Using again Lemma~\ref{lemma:div-preservation}, if~$\divx \v = 0$, then~$\divy \widehat{\v} = 0$ and~$\Piht \v \in \Zht$.

We conclude with an inverse estimate for functions in~$\Vht$.
\begin{lemma}[Inverse estimate]\label{lemma:inverse-estimate}
    For all~$\vht \in \Vht$ and~$n \in \{1, \ldots, N\}$, it holds
    \begin{equation*}
        \Norm{\Dt \vht}{L^2(\L^2; \Qn)} \lesssim ( \tau_n^{-1} + 1) \Norm{\vht}{L^2(\L^2; \Qn)}.
    \end{equation*}
\end{lemma}
\begin{proof}
Let~$\vht \in \Vht$. Using the triangle inequality and the identity in Lemma~\ref{lemma:Dt-phi}, we have
\begin{equation*}
    \Norm{\Dt \vht}{L^2(\L^2; \Qn)} \le \Norm{\phi (\dpt \hv_{h\tau})}{L^2(\L^2; \Qn)} + \Norm{\nablax \bw - \divx \bw \mathbb{I}_d}{\L^{\infty}(\QT)} \Norm{ \phi \hv_{h\tau}}{L^2(\L^2; \Qn)}.
\end{equation*}
Using~\eqref{eq:equiv-vh-3}, the regularity of $\bw$, and the polynomial inverse estimate for functions in~$\Vhk \otimes \Pp{\ell}{\In}$, we get
\begin{equation*}
    \Norm{\Dt \vht}{L^2(\L^2; \Qn)} \lesssim \tau_n^{-1} \Norm{\hv_{h\tau}}{L^2(\In; \L^2(\Omega_0))} + \Norm{\vht}{L^2(\L^2; \Qn)} \lesssim (\tau_n^{-1} + 1) \Norm{\vht}{L^2(\L^2; \Qn)},
\end{equation*}
where, in the last step, we have used again~\eqref{eq:equiv-vh-3}. 
\end{proof}
\subsection{A priori error estimates}
We introduce the tensor-valued field~$\Theta : \Omega \times [0, T] \to \R^{d \times d}$, defined as
\begin{equation}\label{eq:theta}
    \Theta(\by,t) := (\det J_t(\by))^{-1} J_t(\by)^{\top} \; J_t(\by) \, \qquad \forall (\by, t) \in \Omega_0 \times [0, T],
\end{equation}
which is clearly symmetric and, due to Assumption~\ref{asm:At} on the regularity of the ALE map~$\A$, also uniformly strictly positive definite.

Let~$\u$ be the continuous weak solution to~\eqref{eq:model-problem-ALE}, which we assume to satisfy $\widehat{\u} \in H^1(0,T;\H^1(\Omega_0))$ and $\widehat{\u} \in L^2(0,T; \H^{s}(\Omega_0))$ with $s > 3/2$.
Furthermore, let~$\uht \in \Zht$ be the solution to the discrete space--time formulation~\eqref{method:Z}. 
We define the following error functions:
\begin{equation*}\label{eq:error-split}
    \eu := \u - \uht = (\u - \Piht \u) + (\Piht \u - \uht) =: \epi + \Piht \eu. 
\end{equation*}
Due to the consistency of~\eqref{method:Z}, the following error equation holds:
\begin{equation}\label{eq:error-equation}
    \Bht(\Piht \eu, \zht) =  - \Bht(\epi, \zht) \quad \forall \zht \in \Zht. 
\end{equation}
\paragraph{Simplified error equation.}
We now use the properties of the space--time projection~$\Piht$ to simplify the error equation~\eqref{eq:error-equation}.
Using the discrete Reynolds identity in~\eqref{eq:first} and the identity
\begin{equation*}
    \jump{\u \cdot \v}_n = \jump{\u}_n \cdot \v(\tn^+) + \u(\tn^-) \cdot \jump{\v}_n, \qquad \text{for~$n = 1, \ldots, N - 1$},
\end{equation*}
we obtain
\begin{alignat}{3}
    \nonumber
    &\sum_{n = 1}^N \int_{\In}  (\Dt (\Piht \u - \u), \zht)_{\Omega_h(t)} \dt  + \sum_{n = 1}^{N - 1} \big(\jump{\Piht \u - \u}_n, \zht(\cdot, \tn^+)\big)_{\Omega(\tn)} + \big((\Piht \u - \u)(\cdot, 0), \zht(\cdot, 0)\big)_{\Omega_0} \\
    \nonumber
    & = \big((\Piht \u - \u)(\cdot, T), \zht(\cdot, T)\big)_{\Omega(T)} - \sum_{n = 1}^{N - 1} \int_{\Omega(\tn)} \Big( \jump{(\Piht \u - \u) \cdot \zht}_n - \jump{\Piht \u - \u}_n \cdot \zht(\cdot, \tn^+) \Big)\dx \\
    \nonumber
    & \quad - \sum_{n = 1}^N \int_{\In} \big(\Piht \u - \u,  \Dt \zht \big)_{\Omega_h(t)} \dt + \cht(\bw; \Piht \u - \u, \zht) + \cht(\bw; \zht, \Piht \u - \u) \\
    \nonumber
    & = - \big((\Piht \u - \u)(\cdot, T), \zht(\cdot, T)\big)_{\Omega(T)} - \sum_{n = 1}^{N - 1} \big((\Piht \u - \u)(\cdot, \tn^-), \jump{\zht}_n\big)_{\Omega(\tn)} \\
    \nonumber
    & \quad - \sum_{n = 1}^N \int_{\In} \big(\Piht \u - \u, \Dt \zht \big)_{\Omega_h(t)} \dt + \cht(\bw; \Piht \u - \u, \zht) + \cht(\bw; \zht, \Piht \u - \u).
\end{alignat}
Then, using the interpolation property of~$\Pt$ and adding and subtracting suitable terms, the right-hand side of~\eqref{eq:error-equation} can be simplified as follows:
\begin{alignat}{3}
    \nonumber
    -\Bht(\epi, \zht) & = -\big(\phi(\IBDM \hu - \hu)(\cdot, T), \zht(\cdot, T)\big)_{\Omega(T)} - \sum_{n = 1}^{N - 1} \big(\phi(\IBDM \hu - \hu)(\cdot, \tn^-), \jump{\zht}_n\big)_{\Omega(\tn)} \\
    \nonumber
    & \quad - \sum_{n = 1}^N \int_{\In} \big(\phi\big(\IBDM\hu - \hu), \Dt \zht\big)_{\Omega_h(t)} \dt \\
    \nonumber
    & \quad - \sum_{n = 1}^N \int_{\In} \big(\phi\big((\Pt \circ \IBDM)\hu - \IBDM\hu), \Dt \zht\big)_{\Omega_h(t)} \dt \\
    \label{eq:aux-error-equation}
    & \quad - \nu \aht(\Piht \u - \u, \zht) + \cht(\bw; \zht, \Piht \u - \u).
\end{alignat}
Finally, we can use the continuity in time of~$\u$ and the discrete Reynolds identity in~\eqref{eq:first} to simplify the first three terms on the right-hand side of~\eqref{eq:aux-error-equation}. The error equation~\eqref{eq:error-equation} reduces to
\begin{equation}\label{eq:error.equation.final}
    \begin{aligned}
        \Bht(\Piht \eu, \zht) & = \sum_{n = 1}^N \int_{\In} \big(\Dt \phi(\IBDM \hu - \hu), \zht \big)_{\Omega_h(t)} \dt + (\IBDM \u_0 - \u_0, \zht(\cdot, 0))_{\Omega_0} \\
        & \quad - \sum_{n = 1}^N \int_{\In} \big(\phi\big((\Pt \circ \IBDM)\hu - \IBDM\hu), \Dt \zht\big)_{\Omega_h(t)} \dt \\
        & \quad - \cht\big(\bw; \phi(\IBDM \hu - \hu), \zht\big) + \cht\big(\bw; \zht, \phi((\Pt \circ \IBDM) \hu - \IBDM \hu)\big) \\
        & \quad - \nu \aht(\Piht \u - \u, \zht) \\
        & =: M_1(\zht) + M_2(\zht) + M_3(\zht) + M_4(\zht) + M_5(\zht) + M_6(\zht).
    \end{aligned}
\end{equation}
\paragraph{Error estimates.}
We are now in a position to prove the main result in this section.
\begin{theorem}[Error estimate for the discrete error]\label{theo:main:conv}
    Let~$\u$ be the continuous weak solution to~\eqref{eq:model-problem-ALE}, which we assume to satisfy~$\hu \in H^{\ell+1}(0, T; \H^2(\Omega_0)) \cap H^1(0, T; \H^{k+1}(\Omega_0))$. 
    Let also~$\uht \in \Zht$ be the solution to the discrete formulation~\eqref{method:Z}. For~$m \in \{1, \ldots, N\}$, the discrete error~$\Piht \eu$ satisfies
    \begin{equation*}
        \begin{aligned}
            & \Norm{\Piht \eu(\cdot, t_m^-)}{\L^2(\Omega(t_m))}^2 
              + \sum_{n = 1}^{m -1} \Norm{\jump{\Piht \eu}_n}{\L^2(\Omega(\tn))}^2 \\
            & \quad + \Norm{\Piht \eu(\cdot, 0)}{\L^2(\Omega_0)}^2 
              + \nu \sum_{n = 1}^m \int_{\In} \Norm{\Piht \eu(\cdot, t)}{1,h,t}^2 \dt 
              \lesssim \mathcal{E}_{\tau}(\hu) + \mathcal{E}_{h}(\hu, \u_0),
        \end{aligned}
    \end{equation*}
    where the temporal and spatial approximation bounds are given by
    \begin{align*}
        \mathcal{E}_{\tau}(\hu) & := \tau^{2(\ell+1)} \Big[ (\nu + \nu^{-1}) \Norm{\dpt^{(\ell + 1)} \hu}{L^2(0, T; \L^2(\Omega_0))}^2 
          + (\nu + \nu^{-1} h^2) \semiNorm{\dpt^{(\ell+1)} \hu}{L^2(0, T; \H^1(\Omega_0))}^2 \\
        & \hspace{2cm} + \nu \semiNorm{\dpt^{(\ell + 1)}\hu}{L^2(0, T; \H^2(\Omega_0))}^2 \Big], \\
        \mathcal{E}_{h}(\hu, \u_0) & := h^{2k} \Big[ h^2 \nu^{-1} \semiNorm{\dpt \hu}{L^2(0, T; \H^{k+1}(\Omega_0))}^2 
          + \big(\nu + \nu^{-1}(1 + h^2)\big) \semiNorm{\hu}{L^2(0, T; \H^{k+1}(\Omega_0))}^2 \\
        & \hspace{2cm} + h^2\semiNorm{\u_0}{\H^{k+1}(\Omega_0)}^2 
          + \nu T \semiNorm{\hu}{L^{\infty}(0, T; \H^{k+1}(\Omega_0))}^2 \Big].
    \end{align*}
\end{theorem}
\begin{proof}
    Without loss of generality, we prove the result for~$m = N$. 
    We now estimate each term~$M_i(\zht)$, $i = 1, \ldots, 6$,  appearing in~\eqref{eq:error.equation.final} separately.
    
    \paragraph{Estimate of~$M_1$.} 
    Using the identity in Lemma~\ref{lemma:Dt-phi}, together with  the Cauchy--Schwarz inequality and~\eqref{eq:equiv-vh-3}, we get
    \begin{alignat*}{3}
        M_1 (\zht) 
        & \le \sum_{n = 1}^N \int_{\In} \Big(\Norm{\phi_t (\IBDM  \dpt \hu - \dpt \hu)}{\L^2(\Omega_h(t))}  \\
        & \quad + \Norm{\nablax \bw - \divx \bw \mathbb{I}_d}{\L^{\infty}(\Omega(t))} \Norm{\phi_t (\IBDM \hu - \hu)}{\L^2(\Omega_h(t))} \Big)\Norm{\zht}{\L^2(\Omega_h(t))} \dt \\ 
        &\lesssim \sum_{n = 1}^N \int_{\In} \big(\Norm{\IBDM  \dpt \hu - \dpt \hu}{\L^2(\Omega_0)} + \Norm{\IBDM \hu - \hu}{\L^2(\Omega_0)}\big) \Norm{\zht}{\L^2(\Omega_h(t))} \dt \\ 
        &\lesssim h^{k+1} \sum_{n = 1}^N \int_{\In} \big(\semiNorm{\dpt \hu}{\H^{k+1}(\Omega_0)} + \semiNorm{\hu}{\H^{k+1}(\Omega_0)}\big) \Norm{\zht}{\L^2(\Omega_h(t))} \dt,
    \end{alignat*}
    where, in the last inequality, we have applied the approximation properties of~$\IBDM$ from Lemma~\ref{lemma:estimates-IBDM}. Multiplying and dividing by $\nu^{1/2}$, and finally applying the Poincar\'e--Friedrich inequality from Lemma~\ref{lem:dpi}, yields
    \begin{equation*}
        M_1 (\zht) \lesssim \nu^{-1/2}h^{k+1} \big(\semiNorm{\dpt \hu}{L^2(0, T; \H^{k+1}(\Omega_0))}  + \semiNorm{\hu}{L^2(0, T; \H^{k+1}(\Omega_0))} \big) \Big(\nu \int_0^T \Norm{\zht}{1, h, t}^2 \dt \Big)^{1/2}.
    \end{equation*}
    \paragraph{Estimate of~$M_2$.} 
    An immediate application of the Cauchy--Schwarz inequality and the approximation properties in Lemma~\ref{lemma:estimates-IBDM} of~$\IBDM$ gives
    \begin{equation*}
        M_2 (\zht) 
        \lesssim h^{k+1} \semiNorm{\u_0}{\H^{k+1}(\Omega_0)} \Norm{\zht(\cdot, 0)}{\L^2(\Omega_0)}.
    \end{equation*}
    \paragraph{Estimate of~$M_3$.} 
    We first split the term~$M_3$ using Lemma~\ref{lemma:Dt-phi} as follows:
    \begin{alignat}{3}
        \nonumber
        M_3 (\zht) & = - \sum_{n = 1}^N \int_{\In} \Big(\phi\big((\Pt \circ \IBDM) \hu - \IBDM \hu \big) , \, \phi(\dpt \widehat{\z}_{h\tau}) + (\Nablax \bw - \divx \bw \mathbb{I}_d) \zht \Big)_{\Omega_h(t)} \dt \\
        \label{eq:split-M3}
        & =: M_3^{(a)} (\zht) + M_3^{(b)} (\zht).
    \end{alignat}
    As for the first term on the right-hand side, we first make a change of variable, and then use the definition in~\eqref{eq:theta} of~$\Theta$ and the orthogonality properties of~$\Pt$ to obtain
    \begin{alignat*}{3}
        M_3^{(a)} (\zht) & = - \sum_{n = 1}^N \int_{\In} \int_{\Omega_0} \Big((\Pt \circ \IBDM )\hu - \IBDM \hu \Big) \cdot \Theta \, \dpt \widehat{\z}_{h\tau} \dy \dt  \\
        & = - \sum_{n = 1}^N \int_{\In} \int_{\Omega_0} \Big((\Pt \circ \IBDM )\hu - \IBDM \hu \Big) \cdot (\Theta  - \Pi_0^t \Theta) \, \dpt \widehat{\z}_{h\tau}\dy \dt ,
    \end{alignat*}
    where~$\Pi_0^t$ denotes the~$L^2(0, T)$-orthogonal projection into the space of constant functions in time. This identity, combined with the H\"older inequality, the approximation properties in Lemma~\ref{lemma:estimates-Pt} of~$\Pt$, a standard polynomial inverse estimate, and bound~\eqref{eq:equiv-vh-3}, gives
    \begin{alignat}{3}
        \nonumber
        M_3^{(a)} (\zht)& \lesssim \Norm{\Theta - \Pi_0^t \Theta}{L^{\infty}(\In; \L^{\infty}(\Omega_0))}\sum_{n = 1}^N \Norm{(\Pt \circ \IBDM )\hu - \IBDM \hu }{L^2(\In; \L^2(\Omega_0))} \Norm{\dpt \widehat{\z}_{h\tau}}{L^2(\In; \L^2(\Omega_0))} \\
        \nonumber
        & \lesssim \tau_n \Norm{\Theta}{W^{1,\infty}(\In; \L^{\infty}(\Omega_0))}
        \sum_{n = 1}^N \Norm{(\Pt \circ \IBDM )\hu - \IBDM \hu }{L^2(\In; \L^2(\Omega_0))}
        \tau_n^{-1} \Norm{\widehat{\z}_{h\tau}}{L^2(\In; \L^2(\Omega_0))}
        \\
        \nonumber
        & \lesssim \sum_{n = 1}^N \Norm{(\Pt \circ \IBDM )\hu - \IBDM \hu}{L^2(\In; \L^2(\Omega_0))} \Norm{\widehat{\z}_{h\tau}}{L^2(\In; \L^2(\Omega_0))} \\
        \label{eq:estimate-M3-a}
        & \lesssim \tau^{\ell+1} \sum_{n = 1}^N \Norm{\dpt^{(\ell+1)} \IBDM \hu}{L^2(\In; \L^2(\Omega_0))} \Norm{\zht}{\L^2(\Qn)}.
    \end{alignat}
    The term~$M_3^{(b)}$ can be estimated similarly using the Cauchy--Schwarz inequality, the regularity of~$\bw$, bound~\eqref{eq:equiv-vh-3}, and the the approximation properties of~$\Pt$ in Lemma~\ref{lemma:estimates-Pt} as follows:
    \begin{equation}\label{eq:estimate-M3-b}
        \begin{aligned}
            M_3^{(b)} (\zht) & \lesssim \sum_{n = 1}^N \Norm{\phi\Big( (\Pt \circ \IBDM )\hu - \IBDM \hu\Big) }{\L^2(\Qn)} \Norm{\zht}{\L^2(\Qn)} \\
            &\lesssim \tau^{\ell+1} \sum_{n = 1}^N \Norm{\dpt^{(\ell + 1)} \IBDM \hu}{L^2(\In; \L^2(\Omega_0))} \Norm{\zht}{\L^2(\Qn)}.
        \end{aligned}
    \end{equation}
    Combining estimates~\eqref{eq:estimate-M3-a} and~\eqref{eq:estimate-M3-b} with~\eqref{eq:split-M3}, and using the triangle inequality, the approximation properties of~$\IBDM$ in Lemma~\ref{lemma:estimates-IBDM} with $m=1$, and the Poincar\'e--Friedrich inequality from Lemma~\ref{lem:dpi}, implies
    \begin{equation*}
        M_3(\zht) \lesssim \nu^{-1/2}\tau^{\ell + 1} \Big(\Norm{\dpt^{(\ell + 1)}{\hu}}{L^2(0, T; \L^2(\Omega_0))}  + h\semiNorm{\dpt^{(\ell + 1)}\hu}{L^2(0, T; \H^1(\Omega_0))}\Big) \Big(\nu \int_0^T \Norm{\zht}{1, h, t}^2 \dt \Big)^{1/2}.
    \end{equation*}
    \paragraph{Estimate of~$M_4$.}
    Recalling the definition in~\eqref{eq:cht} of the form~$\cht(\cdot; \cdot, \cdot)$, we can split the term~$M_4$ as
    \begin{equation}\label{eq:aux-M4}
        \begin{aligned}
            M_4 (\zht) 
            & = -\sum_{n = 1}^N\int_{\In} \bigg[\big( (\Nablax \phi( \IBDM \hu - \hu)) \bw , \zht \big)_{\Omega_h(t)}  \\ 
            &\quad - \sum_{F \in \Fho} \big((\bw \cdot \nFt) \jump{\phi(\IBDM \hu - \hu)}_{\Ft}, \av{\zht}_{\Ft} \big)_{\Ft}\bigg] \dt.
        \end{aligned}
    \end{equation}
    Using the H\"older inequality,  bound~\eqref{eq:equiv-vh.1},  the approximation properties of~$\IBDM$ in Lemma~\ref{lemma:estimates-IBDM}, and the regularity of~$\bw$, for all~$t \in [0, T]$, we obtain
    \begin{equation}\label{eq:M4-a}
        \begin{aligned}
            \big((\Nablax \phi_t & (\IBDM \hu - \hu)) \bw, \zht\big)_{\Omega_h(t)} \\
            & \le \Norm{\bw}{\L^{\infty}(\QT)}\Norm{\Nablax \phi_t(\IBDM \hu - \hu) }{\L^2(\Omega_h(t))} \Norm{\zht}{\L^2(\Omega(t))} \\
            & \lesssim \Norm{\bw}{\L^{\infty}(\QT)} \big(\Norm{\IBDM \hu - \hu}{\L^2(\Omega_0)} + \Norm{\Nablay (\IBDM \hu - \hu)}{\L^2(\Omega_h(0))} \big) \Norm{\zht}{\L^2(\Omega(t))} \\
            & \lesssim h^k \semiNorm{\hu}{\H^{k+1}(\Omega_0)} \Norm{\zht}{\L^2(\Omega(t))}.
        \end{aligned}
    \end{equation}
    Similarly, now using bounds~\eqref{eq:trace.NEW} and~\eqref{eq:trace.NEW.NEW}, 
    Lemma~\ref{lemma:estimates-IBDM}, and a standard continuous trace inequality, for all~$t \in [0, T]$, we get
    \begin{equation}\label{eq:M4-b}
        \begin{aligned}
             \sum_{F \in \Fho} \big((\bw  \cdot \nFt) & \jump{\phi_t(\IBDM \hu - \hu)}_{\Ft}, \av{\zht}_{\Ft} \big)_{\Ft} \\
             & \lesssim \Norm{\bw}{\L^{\infty}(\QT)} \sum_{K \in \ThO} \Norm{\phi_t(\IBDM \hu - \hu)}{\L^2(\partial K_t)}  \Norm{\zht}{\L^2(\partial K_t)} \\
             & \lesssim \sum_{K \in \ThO} \hK^{-1/2} \Norm{\IBDM \hu - \hu}{\L^2(\partial K)} \Norm{\zht}{\L^2(\Kt)} \\
             & \lesssim h^k \semiNorm{\hu}{\H^{k+1}(\Omega_0)} \Norm{\zht}{\L^2(\Omega(t))}.
        \end{aligned}
    \end{equation}
    Combining~\eqref{eq:M4-a} and~\eqref{eq:M4-b} with~\eqref{eq:aux-M4}, and applying Lemma~\ref{lem:dpi}, we can estimate~$M_4$ as follows:
    \begin{equation*}%
        M_4 (\zht)\lesssim \nu^{-1/2}h^k \semiNorm{\hu}{L^2(0, T; \H^{k+1}(\Omega_0))}  \Big(\nu \int_0^T \Norm{\zht}{1, h, t}^2 \dt \Big)^{1/2}.
    \end{equation*}
    \paragraph{Estimate of~$M_5$.} 
    As for the term~$M_5$, we use the definition in~\eqref{eq:cht} of the form~$\cht(\cdot; \cdot, \cdot)$ to get
    \begin{equation}\label{eq:aux-M5-identity}
        \begin{aligned}
            M_5 (\zht)
            & = \sum_{n = 1}^N \int_{\In} \bigg[\big((\Nablax \zht) \bw, \phi((\Pt \circ \IBDM) \hu - \IBDM \hu) \big)_{\Omega_h(t)} \\
            & \quad - \sum_{F \in \Fho} \big((\bw \cdot \nFt) \jump{\zht}_{\Ft}, \av{\phi((\Pt \circ \IBDM) \hu - \IBDM \hu)}_{\Ft} \big)_{\Ft}\bigg] \dt.
        \end{aligned}
    \end{equation}
    Using the H\"older inequality, bound~\eqref{eq:equiv-vh-3}, and the approximation properties in Lemmas~\ref{lemma:estimates-Pt} and~\ref{lemma:estimates-IBDM} of~$\Pt$ and~$\IBDM$, respectively, for~$n \in \{1, \ldots, N\}$, we obtain
    \begin{equation}\label{eq:M5-a}
        \begin{aligned}
            \int_{\In} &\Big((\Nablax \zht) \bw,  \phi \big((\Pt \circ \IBDM) \hu - \IBDM \hu\big) \Big)_{\Omega_h(t)} \dt \\
            & \le \Norm{\bw}{\L^{\infty}(\QT)} \nu^{1/2} \Norm{\Nablax \zht}{L^2(\In; \L^2(\Omega_h(t)))} \nu^{-1/2} \Norm{\phi((\Pt \circ \IBDM) \hu - \IBDM \hu)}{L^2(\In; \L^2(\Omega(t)))} \\
            & \lesssim \nu^{1/2} \Norm{\Nablax \zht}{L^2(\In; \L^2(\Omega_h(t)))} \nu^{-1/2} \Norm{(\Pt \circ \IBDM) \hu - \IBDM \hu}{L^2(\In; \L^2(\Omega_0))} \\
            & \lesssim \nu^{-1/2} \tau_n^{\ell+1} \Norm{\IBDM \dpt^{(\ell + 1)} \hu}{L^2(\In; \L^2(\Omega_0))} \nu^{1/2} \Norm{\Nablax \zht}{L^2(\In; \L^2(\Omega_h(t)))} \\
            & \lesssim \nu^{-1/2} \tau_n^{\ell+1} \big(\Norm{\dpt^{(\ell + 1)} \hu}{L^2(\In; \L^2(\Omega_0))}  + h \semiNorm{\dpt^{(\ell + 1)} \hu}{L^2(\In; \H^1(\Omega_0))}\big) \nu^{1/2}
            \int_{I_n} \| \zht^{(n)}\|_{1,h,t}^2 \dt.
        \end{aligned}
    \end{equation}
    In a similar way, now using bound~\eqref{eq:trace.NEW}, the shape-regularity of~$\ThO$, and Lemma~\ref{lemma:estimates-IBDM}, the following estimate holds:
    \begin{equation}\label{eq:M5-b}
        \begin{aligned}
            & \int_{\In} \sum_{F \in \Fho}  \big((\bw \cdot \nFt) \jump{\zht}_{\Ft}, \av{\phi((\Pt \circ \IBDM) \hu - \IBDM \hu)}_{\Ft} \big)_{\Ft} \dt \\
            & \lesssim \Norm{\bw}{\L^{\infty}(\QT)} \int_{\In} \sum_{F \in \Fho} \nu^{1/2} \hF^{-1/2} \Norm{\jump{\zht}_{\Ft}}{\L^2(\Ft)} \nu^{-1/2} \hF^{1/2} \Norm{\av{\phi((\Pt \circ \IBDM) \hu - \IBDM \hu)}_{\Ft}}{\L^2(\Ft)}\dt  \\
            & \lesssim \Big(\nu \int_{\In} \sum_{F \in \Fho} \hF^{-1} \Norm{\jump{\zht}_{\Ft}}{\L^2(\Ft)}^2 \dt \Big)^{1/2}  \Big(\nu^{-1} \int_{\In} \sum_{K \in \ThO} \hK \Norm{\phi((\Pt \circ \IBDM) \hu - \IBDM \hu)}{\L^2(\partial K_t)}^2 \dt \Big)^{1/2} \\
            & \lesssim \Big(\nu \int_{\In} \sum_{F \in \Fho} \hF^{-1} \Norm{\jump{\zht}_{\Ft}}{\L^2(\Ft)}^2 \dt \Big)^{1/2}  \Big(\nu^{-1} \int_{\In} \sum_{K \in \ThO} \hK \Norm{(\Pt \circ \IBDM) \hu - \IBDM \hu)}{\L^2(\partial K)}^2 \dt \Big)^{1/2} \\
            & \lesssim \nu^{-1/2} \tau_n^{\ell+1} \big(\Norm{\dpt^{(\ell + 1)} \hu}{L^2(\In; \L^2(\Omega_0))}  + h \semiNorm{\dpt^{(\ell + 1)} \hu}{L^2(\In; \H^1(\Omega_0))}\big) \nu^{1/2}
            \int_{I_n} \| \zht^{(n)}\|_{1,h,t}^2 \dt.
        \end{aligned}
    \end{equation}
    Combining estimates~\eqref{eq:M5-a} and~\eqref{eq:M5-b} with identity~\eqref{eq:aux-M5-identity} yields
    \begin{equation*}
        M_5(\zht) \lesssim \nu^{-1/2} \tau^{\ell+1} \big(\Norm{\dpt^{(\ell + 1)} \hu}{L^2(0, T; \L^2(\Omega_0))}  + h \semiNorm{\dpt^{(\ell + 1)} \hu}{L^2(0, T; \H^1(\Omega_0))}\big) \Big(\nu \int_0^T \Norm{\zht}{1, h, t}^2 \dt \Big)^{1/2}.
    \end{equation*}
    \paragraph{Estimate of~$M_6$.} 
    Using the definition in~\eqref{eq:aht} of the bilinear form~$\aht(\cdot, \cdot)$, we have
    \begin{equation}\label{eq:identity-M6}
        \begin{aligned}
            M_6 (\zht)
            & = - \nu \sum_{n = 1}^N \int_{\In} \bigg[\sum_{K \in \ThO} \big(\Nablax \phi((\Pt \circ \IBDM) \hu - \hu), \Nablax \zht \big)_{\Kt} \\
            & \qquad \quad - \sum_{F \in \Fh} \big(\av{\Nablax \phi((\Pt \circ \IBDM) \hu - \hu)}_{\Ft} \nFt, \jump{\zht}_{\Ft} \big)_{\Ft} \\
            & \qquad\quad  - \sum_{F \in \Fh} \big(\jump{\phi((\Pt \circ \IBDM) \hu - \hu)}_{\Ft}, \av{\Nablax \zht}_{\Ft} \nFt \big)_{\Ft} \\
            & \qquad\quad  + \sum_{F \in \Fh} \big(\sigma \hF^{-1} \jump{\phi((\Pt \circ \IBDM) \hu - \hu)}_{\Ft}, \jump{\zht}_{\Ft}\big)_{\Ft} \bigg] \dt \\
            & =: M_6^{(a)} (\zht) + M_6^{(b)} (\zht) + M_6^{(c)} (\zht) + M_6^{(d)}(\zht).
        \end{aligned}
    \end{equation}
    The Cauchy--Schwarz inequality, bound~\eqref{eq:equiv-vh.1}, and the approximation properties in Lemmas~\ref{lemma:estimates-Pt} and~\ref{lemma:estimates-IBDM} of~$\Pt$ and~$\IBDM$, respectively, give
    \begin{equation*}
        \begin{aligned}
            M_6^{(a)} (\zht) & \lesssim \nu^{1/2} \sum_{n = 1}^N \big(\Norm{(\Pt \circ \IBDM) \hu - \hu}{L^2(\In; \L^2(\Omega_0))} + \Norm{\Nablay ((\Pt \circ \IBDM) \hu - \hu)}{L^2(\In; \L^2(\Omega_h(0)))}\big) \\
            & \quad \times \nu^{1/2} \Norm{\Nablax \zht}{L^2(\In; \L^2(\Omega_h(t)))} \\
            & \lesssim \nu^{1/2} \sum_{n = 1}^N \Big(\tau_n^{\ell+1} \big(\Norm{\dpt^{(\ell + 1)} \hu}{L^2(\In; \L^2(\Omega_0))} + \semiNorm{\dpt^{(\ell + 1)} \hu}{L^2(\In; \H^{1}(\Omega_0))} \big) + h^k \semiNorm{\hu}{L^2(\In; \H^{k+1}(\Omega_0))} \Big) \\
            & \quad \times \nu^{1/2} 
            \int_{I_n} \| \zht^{(n)}\|_{1,h,t}^2 \dt \\
            & \lesssim \nu^{1/2} \Big( \tau^{\ell+1} \big(\Norm{\dpt^{(\ell + 1)} \hu}{L^2(0,T; \L^2(\Omega_0))} + \semiNorm{\dpt^{(\ell + 1)} \hu}{L^2(0,T; \H^{1}(\Omega_0))} \big) + h^k  \semiNorm{\hu}{L^2(0,T; \H^{k+1}(\Omega_0))} \Big) \\
            & \quad \times \Big(\nu \int_0^T \Norm{\zht}{1, h, t}^2 \dt \Big)^{1/2}.
        \end{aligned}
    \end{equation*}
    As for the term~$M_6^{(b)} $, the use of the Cauchy--Schwarz inequality, the shape-regularity of~$\ThO$, and bound~\eqref{eq:trace-grad.NEW} leads to
    \begin{alignat*}{3}
        M_6^{(b)}(\zht) & \lesssim \nu \sum_{n = 1}^N \int_{\In} \bigg(\sum_{K \in \ThO} \hK \Norm{\Nablax \phi ((\Pt \circ \IBDM) \hu - \hu) }{\L^2(\partial \Kt)}^2 \bigg)^{1/2} \!\! \bigg(\sum_{F \in \Fh} \hF^{-1} \Norm{\jump{\zht}_{\Ft}}{\L^2(\Ft)}^2 \bigg)^{1/2} \!\!\dt \\
        & \lesssim \nu \sum_{n = 1}^N \int_{\In} \bigg(\sum_{K \in \ThO} \Big(\Norm{(\Pt \circ \IBDM) \hu - \hu}{\L^2(K)}^2 + (\hK^2 + 1) \Norm{\Nablay ((\Pt \circ \IBDM) \hu - \hu)}{\L^2(K)}^2 \\
        & \quad + \hK^2 \Norm{\Nablay^2 ((\Pt \circ \IBDM) \hu - \hu)}{\L^2(K)}^2 \Big) \bigg)^{1/2} \bigg(\sum_{F \in \Fh} \hF^{-1} \Norm{\jump{\zht}}{\L^2(\Ft)}^2 \bigg)^{1/2}\!\! \dt,
    \end{alignat*}
    which, together with the approximation properties in Lemmas~\ref{lemma:estimates-Pt} and~\ref{lemma:estimates-IBDM} of~$\Pt$ and~$\IBDM$, respectively, gives
    \begin{alignat*}{3}
        M_6^{(b)}(\zht) & \lesssim \nu^{1/2} \Big(\tau^{\ell+1}\big(\Norm{\dpt^{(\ell + 1)} \hu}{L^2(0, T; \L^2(\Omega_0))} + \semiNorm{\dpt^{(\ell + 1)} \hu}{L^2(0, T; \H^1(\Omega_0))} + \semiNorm{\dpt^{(\ell + 1)} \hu}{L^2(0, T; \H^2(\Omega_0))} \big) \\
        & \quad + h^k \semiNorm{\hu}{L^2(0, T; \H^{k+1}(\Omega_0))} \Big) \Big(\nu \int_0^T \Norm{\zht}{1, h, t}^2 \dt \Big)^{1/2}
    \end{alignat*}
    For the term~$M_6^{(c)}$, we first observe that~$\phi(\Pt \hu)$ and~$\u = \phi(\hu)$ are continuous across the facets~$\Ft$. Therefore, we use the Cauchy--Schwarz inequality, the~$L^{\infty}(0, T)$ stability of~$\Pt$ (see~\cite[Lemma~4.1]{Beirao-Gomez-Dassi:2025}), bounds ~\eqref{eq:trace.NEW} and ~\eqref{eq:trace-grad.NEW.NEW}, and the approximation properties of~$\IBDM$ combined with a standard trace inequality on each element~$K \in \ThO$ to obtain
    \begin{alignat*}{3}
        M_6^{(c)} (\zht) 
        & \lesssim \nu \sum_{n = 1}^N \sum_{K \in \ThO} \hK^{-1/2} \Norm{\phi((\Pt\circ \IBDM) \hu - \Pt \hu)}{L^2(\In; \L^2(\partial \Kt))} \hK^{1/2} \Norm{\Nablax \zht}{L^2(\In; \L^2(\partial \Kt))} \\
        & \lesssim \nu \sum_{n = 1}^N \tau_n^{1/2} \sum_{K \in \ThO} \hK^{-1/2} \Norm{\IBDM \hu - \hu}{L^{\infty}(\In; \L^2(\partial K))} \Norm{\Nablax \zht}{L^2(\In; \L^2(\Kt))} \\
        & \lesssim  \sqrt{\nu T} h^k \semiNorm{\hu}{L^{\infty}(0, T; \H^{k+1}(\Omega_0))} \Big(\nu \int_0^T\Norm{\zht}{1,h,t}^2 \dt \Big)^{1/2},
    \end{alignat*}
    where, in the last step, we have used again the Cauchy--Schwarz inequality. 
    The term~$M_6^{(d)}$ can be estimated similarly as follows:
    \begin{alignat*}{3}
        M_6^{(d)} (\zht) 
        & \lesssim \nu \sum_{n = 1}^N \int_{\In} \bigg(\sum_{K \in \ThO} \hK^{-1} \Norm{\phi((\Pt \circ \IBDM ) \hu - \Pt \hu)}{\L^2(\partial \Kt)}^2\bigg)^{1/2} \!\! \bigg(\sum_{F \in \Fh} \hF^{-1} \Norm{\jump{\zht}_{\Ft}}{\L^2(\Ft)}^2 \bigg)^{1/2} \!\! \dt \\
        & \lesssim \sqrt{\nu T} h^{k} \semiNorm{\hu}{L^{\infty}(0, T; \H^{k+1}(\Omega_0))} \bigg(\nu \int_0^T \Norm{\zht}{1, h, t}^2 \dt \bigg)^{1/2}.
    \end{alignat*}
    Combining the estimates for $M_6^{(a)}$--$M_6^{(d)}$ with \eqref{eq:identity-M6} leads to
    \begin{equation*}
        \begin{aligned}
            M_6 (\zht) &\lesssim \Big[ \nu^{1/2} \Big(\tau^{\ell+1}\big(\Norm{\dpt^{(\ell + 1)} \hu}{L^2(0, T; \L^2(\Omega_0))} + \semiNorm{\dpt^{(\ell + 1)} \hu}{L^2(0, T; \H^1(\Omega_0))} + \semiNorm{\dpt^{(\ell + 1)} \hu}{L^2(0, T; \H^2(\Omega_0))} \big) \\
            & \quad + h^k \semiNorm{\hu}{L^2(0, T; \H^{k+1}(\Omega_0))} \Big) + \sqrt{\nu T} h^{k} \semiNorm{\hu}{L^{\infty}(0, T; \H^{k+1}(\Omega_0))} \Big] \bigg(\nu \int_0^T \Norm{\zht}{1, h, t}^2 \dt \bigg)^{1/2}.
        \end{aligned}
    \end{equation*}
    Now plugging the estimates for $M_1$--$M_6$ into \eqref{eq:error.equation.final}, after basic algebraic manipulations, we get
    \begin{equation}\label{eq:err.est.consistency}
        \begin{aligned}
            \Bht(\Piht \eu, \zht)
            &\lesssim \Bigg[ h^{2 k} \Big( 
              h^2 \nu^{-1} \semiNorm{\dpt \hu}{L^2(0, T; \H^{k+1}(\Omega_0))}^2  + ( \nu + (1+h^2) \nu^{-1} )\semiNorm{\hu}{L^2(0, T; \H^{k+1}(\Omega_0))}^2  \\ 
             &\quad + h^2 \semiNorm{\u_0}{\H^{k+1}(\Omega_0)}^2  + \nu T \semiNorm{\hu}{L^{\infty}(0, T; \H^{k+1}(\Omega_0))}^2 \Big) \\ 
            &\quad + \tau^{2(\ell+1)} \Big(
             (\nu^{-1} + \nu) \Norm{\dpt^{(\ell + 1)} \hu}{L^2(0, T; \L^2(\Omega_0))}^2  + (\nu^{-1} h^2  + \nu) \semiNorm{\dpt^{(\ell + 1)} \hu}{L^2(0, T; \H^1(\Omega_0))}^2 \\ 
             &\quad + \nu \semiNorm{\dpt^{(\ell + 1)} \hu}{L^2(0, T; \H^2(\Omega_0))}^2 \Big) \Bigg]^{1/2} \Bigg[\nu \int_0^T \Norm{\zht}{1, h, t}^2 \dt  +  \Norm{\zht(\cdot, 0)}{\L^2(\Omega_0)}^2
            \Bigg]^{1/2}.
        \end{aligned}
    \end{equation}
    On the other hand, applying the same argument that leads to~\eqref{eq:node.stab.bound.lhs.bound}, we get
    \begin{equation}\label{eq:err.est.coercivity}
        \begin{aligned}
            \Bht(\Piht \eu, \Piht \eu) &\gtrsim
            \Norm{\Piht \eu(\cdot, T)}{\L^2(\Omega(T))}^2 +\sum_{n = 1}^{N-1} \Norm{\jump{\Piht \eu}_n}{\L^2(\Omega(t_n))}^2 + \Norm{\Piht \eu(\cdot, 0)}{\L^2(\Omega_0)}^2 \\ 
            &\quad + \sum_{n = 1}^{N} \int_{\In} 
            \nu \|\Piht \eu \|_{1,h,t}^2
            \dt.
        \end{aligned}
    \end{equation}
    Substituting $\zht = \Piht \eu$ in~\eqref{eq:err.est.consistency}, combining the resulting bound with~\eqref{eq:err.est.coercivity} and applying the Young  inequality to absorb the terms involving~$\Piht \eu$ on the right-hand side yields, after basic algebraic manipulations, the desired result.
\end{proof}
The following corollary follows easily combining Theorem~\ref{theo:main:conv} with the triangle inequality and approximation estimates for the operator~$\Piht$.
\begin{corollary}[Error estimate]\label{cor:u-minus-uh.error.est}
    Under the same assumptions of Theorem \ref{theo:main:conv}, for~$m \in \{1,2,..,N\}$, it holds
    $$
        \Norm{\u(\cdot, t_m^-) - \uht(\cdot, t_m^-)}{\L^2(\Omega(t_m))}^2  
        + \nu \sum_{n = 1}^m \int_{\In} \Norm{\u(\cdot, t) - \uht}{1,h,t}^2 \dt 
          \lesssim (1 + \nu + \nu^{-1}) \, (h^{2k} + \tau^{2(\ell+1)}) \, .
    $$
\end{corollary}
\begin{remark}[Pressure-robustness]
    The pressure-robustness of the scheme is reflected by the fact that the pressure solution does not enter the velocity error estimates in Theorem~\ref{theo:main:conv} and Corollary~\ref{cor:u-minus-uh.error.est}.
    \eremk
\end{remark}
\section{The low-order cases} \label{sec:6}
In the present section, we investigate more specifically the low-order (in time) cases $\ell=0$ and~$\ell = 1$, showing, in particular, stability and error estimates that are robust in the viscosity parameter $\nu$.
\begin{proposition}[Stability for the low-order cases]\label{prop:los}
    Let the polynomial degree in time $\ell \in \{0,1\}$. Then the following improved stability estimate also holds:
    $$
        \Norm{\uht}{L^\infty(\L^2;\QT)}^2 + \sum_{n = 1}^{N} \int_{\In} 
        \nu \|\uht (\cdot,t)\|_{1,h,t}^2
        \lesssim
        \Norm{\f (\cdot,t)}{L^1(\L^2;\QT)}^2 + \Norm{\u_0}{\L^2(\Omega_0)}^2 \, ,
    $$
    where the hidden constant is independent of $\nu$.
\end{proposition}
\begin{proof}
    Since the case~$\ell = 0$ is simpler, we focus on the case~$\ell = 1$.
    We start by a simple but crucial calculation (similar arguments in simpler settings with fixed domains can be found in \cite{hansbo1990velocity,beirao2024pressure}).
    On a time interval $I_n$, $n=1,2,\ldots,N$, choose $\{ \psi_j^{(n)} \}_{j=0}^1$ as the Lagrange basis associated with the nodes~$t_{n-1}$ and~$t_n$. We thus have, by definition,
    $$
        \begin{aligned}
            & \uht(\bx, t) = \sum_{j = 0}^{1} (\phi \ba_{h, j})(\bx) \psi_j^{(n)}(t), \ 
            \textrm{ with } \\ 
            & \uht(\bx, t_{n-1}^{+}) =  (\phi_{t_{n-1}} \ba_{h, 0})(\bx) \ ,
            \quad \uht(\bx, t_{n}^{-}) =  (\phi_{t_{n}} \ba_{h, 1})(\bx) \, .
        \end{aligned}
    $$
    From to the last two identities above and \eqref{eq:equiv-vh-3}, we deduce
    \begin{equation*}
        \| \uht(\cdot, t_{n-1}^{+}) \|_{\L^2(\Omega(t_{n-1}))} \simeq \| \ba_{h, 0} \|_{\L^2(\Omega_0)} \, , \quad 
        \| \uht(\cdot, t_{n}^{-}) \|_{\L^2(\Omega(t_{n}))} \simeq \| \ba_{h, 1} \|_{\L^2(\Omega_0)} \, .
    \end{equation*}
    Combining the above equations with the regularity of the mapping~$\A$, it follows
    \begin{equation}\label{eq:koalino}
        \begin{aligned}
            \| \uht(\cdot, t) \|_{\L^2(\Omega(t))} & \le
            \sum_{j = 0}^{1} \| \phi \ba_{h, j} \|_{\L^2(\Omega(t))}
            \lesssim \sum_{j = 0}^{1} \| \ba_{h, j} \|_{L^2(\Omega_0)} \\
            & \lesssim \| \uht(\cdot, \tnmo^{+}) \|_{\L^2(\Omega(\tnmo))}
            + \| \uht(\cdot, t_{n}^{-}) \|_{\L^2(\Omega(t_{n}))}
            \qquad \textrm{ for all } t \in I_n \, ,
        \end{aligned}
    \end{equation}
    so that the $L^\infty(\L^2;\Qn)$ norm of~$\uht$ is controlled by the values of its $\L^2(\Omega(t))$ norm at the two extrema of~$I_n$.
    The remaining part of the argument takes the steps from the proof of Theorem~\ref{thm:well-pos}.
    By the H\"older inequality, the right-hand side in \eqref{eq:node.stab.bound.initial} can be bounded, for any $m \in \{1,2,\ldots,N\}$ and all~$\varepsilon > 0$, by 
    \begin{equation}\label{eq:low:bound}
        \mathrm{RHS} \leq \frac{1}{2 \varepsilon} \Norm{\f}{L^1(\L^2;\QT)}^2 
        +   \frac{\varepsilon}{2} \Norm{\uht}{L^\infty(\L^2;\QT)}^2  \\
        + \Norm{\u_0}{\L^2(\Omega_0)}^2 + \frac{1}{4}\Norm{\uht(\cdot, 0)}{\L^2(\Omega_0)}^2 .
    \end{equation}
    As a consequence, we can observe that the maximum over all $m \in \{1,2,\ldots,N\}$ of the LHS in~\eqref{eq:aux-LHS-RHS} is controlled by the above bound on RHS.
    Using~\eqref{eq:node.stab.bound.lhs.bound} and the triangle inequality, we obtain 
    $$
        \frac{1}{2} \Norm{\uht(\cdot, t_m^+)}{\L^2(\Omega(t_m))}^2 \le \Norm{\uht(\cdot, t_m^-)}{\L^2(\Omega(t_m))}^2 + \Norm{\jump{\uht}_m}{\L^2(\Omega(t_m))}^2 \, ,
    $$
    which, together with the observation above, easily leads to the bound
    \begin{equation*}
        \begin{aligned}
            \frac{1}{2} \, \max_{m=0,\ldots,N} \ \boldsymbol \Xi_m \
            \le \frac{1}{2 \varepsilon} \Norm{\f}{L^1(\L^2;\QT)}^2 
            +   \frac{\varepsilon}{2} \Norm{\uht}{L^\infty(\L^2;\QT)}^2 
            + \Norm{\u_0}{\L^2(\Omega_0)}^2 + \frac{1}{4}\Norm{\uht(\cdot, 0)}{\L^2(\Omega_0)}^2 \, ,
        \end{aligned}
    \end{equation*}
    with
    $$
        \boldsymbol \Xi_m :=
        \begin{cases}
            \Norm{\uht(\cdot, 0)}{\L^2(\Omega_0)}^2 & \quad \textrm{if } m=0 \, , \\
            \max \big\{ \Norm{\uht(\cdot, t_m^-)}{\L^2(\Omega(t_m))}^2 \, , \ \frac{1}{2} \Norm{\uht(\cdot, t_m^+)}{\L^2(\Omega(t_m))}^2 \big\} & \quad \textrm{if } m=1, 2, \ldots, N-1 \, , \\
            \Norm{\uht(\cdot, t_N^-)}{\L^2(\Omega(t_N))}^2 & \quad \textrm{if } m=N \, .
        \end{cases}
    $$
    Applying \eqref{eq:koalino} and simple manipulations, for all~$\varepsilon > 0$, we obtain 
    $$
        \Norm{\uht}{L^\infty(\L^2;\QT)}^2 
        \le
        C \Big( \frac{1}{2 \varepsilon} \Norm{\f}{L^1(\L^2;\QT)}^2 
        + \frac{\varepsilon}{2} \Norm{\uht}{L^\infty(\L^2;\QT)}^2 
        + \Norm{\u_0}{\L^2(\Omega_0)}^2 \Big) \, ,
    $$
    where the constant $C$ is independent of $\nu$. The bound on~$\Norm{\uht}{L^\infty(\L^2; \QT)}^2$ now follows taking~$\varepsilon$ sufficiently small. 
    In order to recover the analogous inequality also on the term $\sum_{n = 1}^{N} \int_{\In} \nu \|\uht\|_{1,h,t}^2 \dt $, we pursue the following steps: (1) start from \eqref{eq:node.stab.bound.initial} and \eqref{eq:node.stab.bound.lhs.bound} with $m=N$; (2) bound the right-hand side first applying \eqref{eq:low:bound} with $\varepsilon=1$ and afterwards using the bound on $\Norm{\uht}{L^\infty(\L^2;\QT)}^2$.
\end{proof}
We now state (in brief format) also the corresponding convergence result.
\begin{proposition}[Discrete error estimate for the low-order cases]\label{prop:loc}
    Let the polynomial degree in time $\ell \in \{0,1\}$. Then, assuming~$\hu$ sufficiently smooth, 
    the following improved bound for the discrete error also holds:
    \begin{equation*}
        \begin{aligned}
            \Norm{\Piht \eu}{L^\infty(\L^2;\QT)}^2  
            + \nu \sum_{n = 1}^N \int_{\In} \Norm{\Piht \eu(\cdot, t)}{1,h,t}^2 \dt 
              \lesssim (1+\nu) \, (h^k + \tau^{\ell+1}) {\cal R}(\hu),
        \end{aligned}
    \end{equation*}
    where the hidden constant is independent of $\nu$, and~${\cal R}(\hu)$ gathers the regularity terms of the exact solution.
\end{proposition}
\begin{proof}
    We omit the full proof, as it is essentially a combination of the arguments in Theorem \ref{theo:main:conv} and Proposition \ref{prop:los}. In particular, by the same steps in the proof of Proposition \ref{prop:los}, we now have control on~$\Norm{\Piht \eu}{L^\infty(\L^2;\QT)}^2$, which allows us to handle the terms~$M_1$, $M_3$, and~$M_4$ without using the diffusion term on the left-hand side. 
    Therefore, we limit ourselves to show \emph{briefly} an alternative bound for the term $M_5$. 
    By the antisymmetric property of~$\cht(\cdot; \cdot, \cdot)$, we can write
    \begin{alignat*}{3}
        \nonumber
         M_5(\zht)
        & = - \cht(\bw; \phi((\Pt \circ \IBDM) \hu - \IBDM \hu),\zht) \\
        \nonumber
        & = - \sum_{n = 1}^N \int_{\In} \Big[\big(\Nablax \phi((\Pt \circ \IBDM) \hu - \IBDM \hu) \bw, \zht \big)_{\Omega_h(t)} \\
        \nonumber
        & \quad + \sum_{F \in \Fho} \big((\bw \cdot \nFt) \jump{\phi((\Pt \circ \IBDM) \hu - \IBDM \hu)}_{\Ft}, \av{\zht}_{\Ft} \big)_{\Ft}\Big] \dt \\
        & =: M_5^{(a)}(\zht) + M_5^{(b)}(\zht),
    \end{alignat*}
    with $\zht=\Piht \eu$. 
    
    We bound the first term by the H\"older inequality and the boundedness of $\bw$ in $\L^\infty(\QT)$, as follows:
    $$
        M_5^{(a)}(\zht) \lesssim \bigg( 
        \int_0^T \Norm{\Nablax \phi\big((\Pt \circ \IBDM) \hu - \IBDM \hu\big)}{\L^2(\Omega_h(t))} \dt 
        \bigg)  
        \Norm{\zht}{L^{\infty}(\L^2;\QT)} \, .
    $$
    We then apply \eqref{eq:equiv-vh.1} and, as usual, Lemmas~\ref{lemma:estimates-IBDM} and~\ref{lemma:estimates-Pt} (similarly to the last steps in~\eqref{eq:M5-a}) to obtain
    $$
        \begin{aligned}
            M_5^{(a)} (\zht) & \lesssim 
            \bigg(
            \int_0^T \Big( \| (\Id - \Pt) \Nablay \IBDM \hu \|_{\L^2(\Omega_{h}(0))} 
            + \Norm{(\Id - \Pt) \IBDM \hu}{\L^2(\Omega_{0})}\Big) \dt  \bigg) 
            \| \zht \|_{L^{\infty}(\L^2;\QT)} \\
            & \lesssim \tau^{\ell+1}
            \Norm{\dpt^{(\ell+1)} \hu}{L^1(0, T; \H^1(\Omega_0))} \Norm{\zht}{L^{\infty}(\L^2;\QT)} \, .
        \end{aligned}
    $$
    Applying \eqref{eq:trace.NEW.NEW} to $\zht$, the H\"older inequality, standard manipulations, and~\eqref{eq:equiv-vh-2} yield
    $$
        \begin{aligned}
            M_5^{(b)}(\zht) & \lesssim 
            \bigg( \int_0^T \!\! 
            \sum_{F \in \Fho} h_F^{-1/2} \Norm{\jump{\phi((\Pt \circ \IBDM) \hu - \IBDM \hu)}}{\L^2(\Ft)} \dt
            \bigg) 
            \Norm{\zht }{L^{\infty}(\L^2;\QT)} \\
            & \lesssim \bigg( \int_0^T \!\! 
            \sum_{F \in \Fho} h_F^{-1/2} \Norm{\jump{(\Pt - {\cal I}) \circ \IBDM \hu}}{\L^2(F)}
            \bigg)  
            \Norm{\zht}{L^{\infty}(\L^2;\QT)} \, .
        \end{aligned}
    $$
    First writing $\jump{(\Pt - \Id) \circ \IBDM \hu} = (\Pt - \Id) \jump{\IBDM \hu}$, then by
    the continuity in $L^\infty(0, T)$ of the $\Pt$ operator, finally recalling that the jumps of $\hu$ vanish, we obtain 
    $$
        \begin{aligned}
            M_5^{(b)}(\zht) & \lesssim \bigg(\int_0^T 
            \sum_{F \in \Fho} h_F^{-1/2} \Norm{\jump{\hu - \IBDM \hu}}{\L^2(F)} \dt
            \bigg) 
            \Norm{\zht}{L^{\infty}(\L^2;\QT)} \\
            & \lesssim h^k \semiNorm{\hu}{L^{1}(0,T; \H^{k+1}(\Omega_0))} \,
            \Norm{\zht}{L^{\infty}(\L^2;\QT)},
        \end{aligned}
    $$
    where, in the last step, we used trace estimates and Lemma~\ref{lemma:estimates-IBDM}.
\end{proof}
\begin{remark}[$\nu$-robustness for high-order approximations]\label{rem:nu-robustenss-high-degree}
    Extending the above stronger results also to the case~$\ell \ge 2$ would require introducing suitable test functions, analogous to the case of fixed domains (see~\cite{Gomez:2026b,Beirao-Gomez-Dassi:2025}). Additionally, one would need to cope with the specific difficulty of the very strong space--time tangling in the present setting induced by the Piola map. Furthermore, 
    this may also require the combination with suitable discrete integration in time, to be incorporated in the definition of the scheme. Such investigations are beyond the scope of the present work and will be the objective of future studies.
    \eremk
\end{remark}

\section{Numerical tests}\label{sec:numerical-tests}
In this section, we numerically validate the theoretical findings presented in the previous sections. Specifically, we assess the performance of the proposed method through a quantitative convergence and robustness analysis in Sections~\ref{sec:test1_conv_rates}, \ref{sec:test2_nu-robustness}, and~\ref{sec:test3_pressure-robustness}, while a qualitative analysis is carried out in Section~\ref{sec:test4_ovalization}. 

For the quantitative tests, we consider the reference domain $\Omega_0 = (0, 1)^2$, whereas for the qualitative test, the initial domain is defined as the unit disk~$\Omega_0 \coloneqq \{ \by \in \mathbb{R}^2 : |\by| < 1 \}$. 
In all the numerical experiments, we set the time interval to $(0, T)$ with $T=1$, the interior-penalty stabilization parameter to~$\sigma = 10$, and the spatial polynomial degree to~$k=1$. Moreover, we employ uniform time partitions.

The ALE maps considered in the quantitative and qualitative tests are defined as
\begin{equation}\label{eq:ALEMaps-tests}
    \A^{\mathrm{S}} (\by , t) \coloneqq \left( y_1(1+ty_2) ,\,  y_2 \right) \quad \text{and} \quad \A^{\mathrm{D}} (\by , t) \coloneqq \left( (1+2t)y_1,\, \frac{1}{1+2t} y_2 \right),
\end{equation}
respectively (see Figure~\ref{fig:ALE_map_square_circle} for a graphical representation).
\begin{figure}
    \centering
    \begin{subfigure}[b]{0.45\textwidth}
        \centering
            \includegraphics[width=\linewidth]{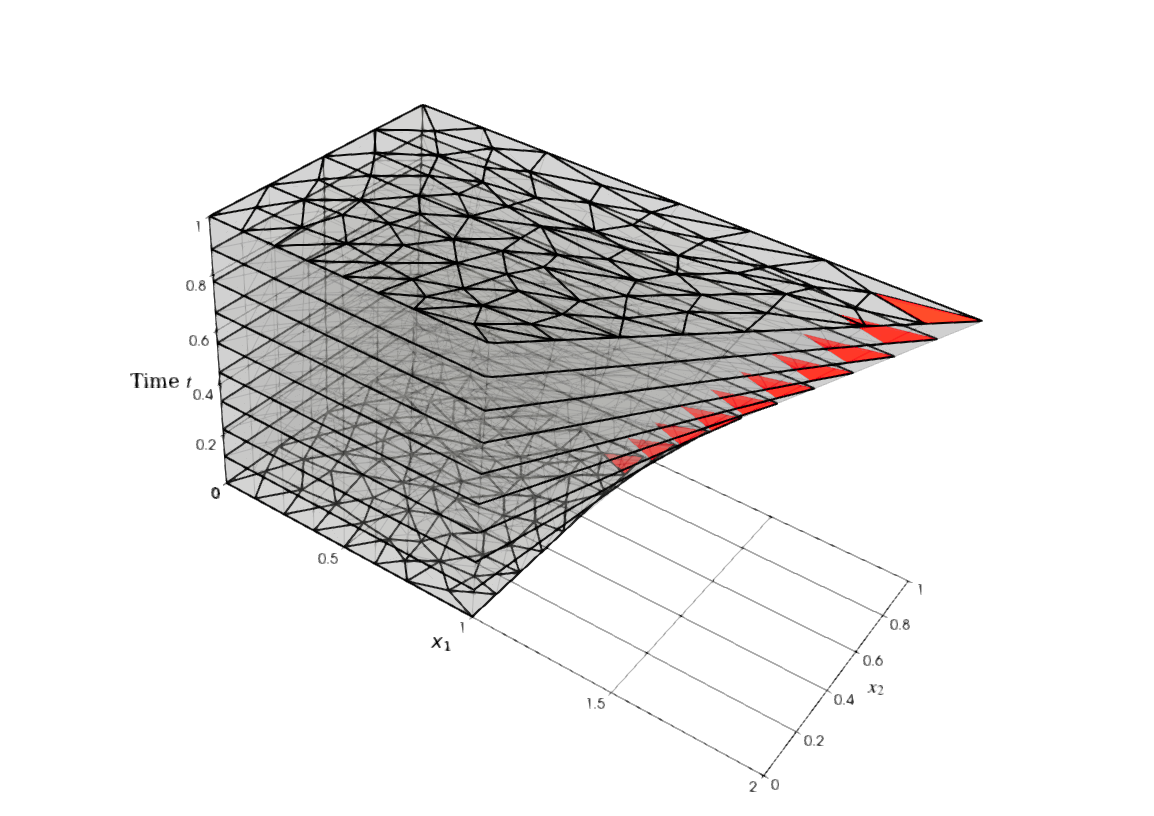}
    \end{subfigure}
    \hfill
    \begin{subfigure}[b]{0.45\textwidth}
        \centering
            \includegraphics[width=\linewidth]{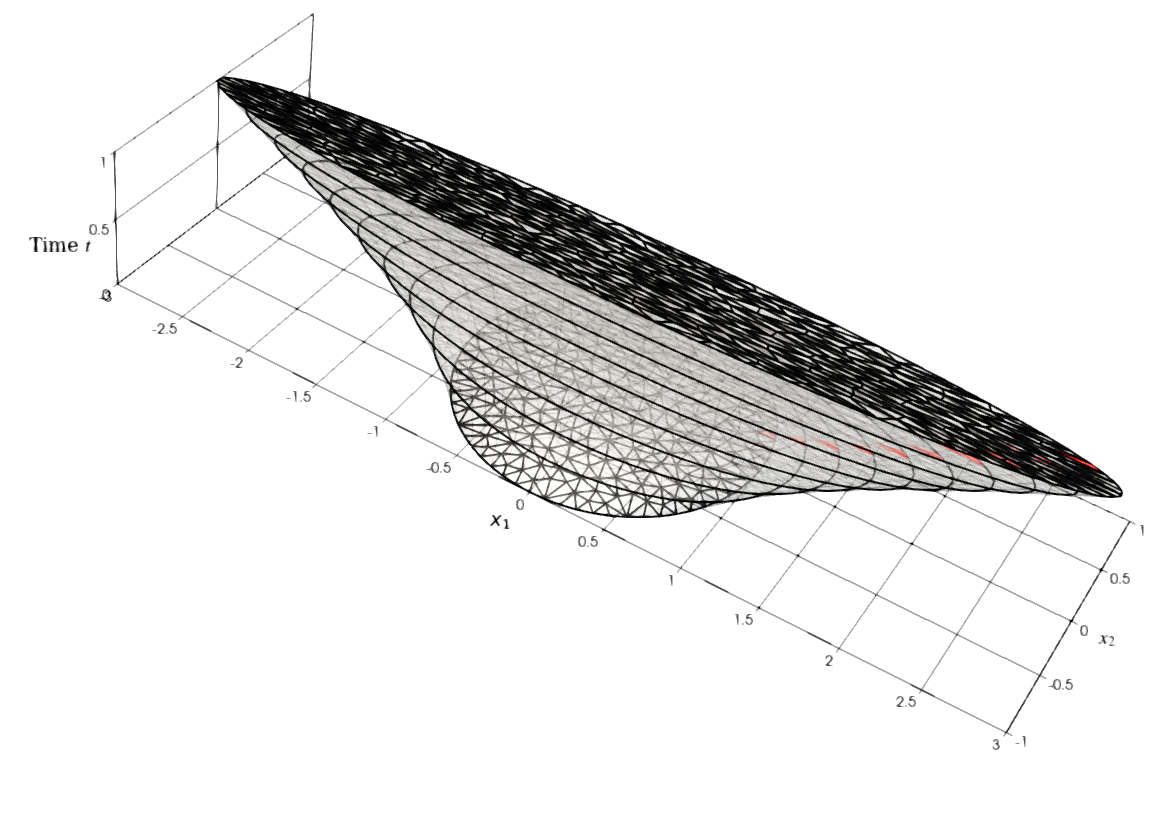}
    \end{subfigure}
    \caption{Space--time evolution of the domain $\Omega_0$. \textbf{Left panel:} unit square with ALE map~$\A^{\mathrm{S}}$. \textbf{Right panel:} unit disk with ALE map~$\A^{\mathrm{D}}$ (cf. \eqref{eq:ALEMaps-tests}).}
    \label{fig:ALE_map_square_circle}
\end{figure}
We consider the following error measures:
\begin{equation*}
    \begin{aligned}
        \mathrm{velERR}_{T} & \coloneqq \Norm{\u (\cdot, T)- \uht (\cdot, T)}{\L^2(\Omega(T))}, \quad \mathrm{preERR}_{T} \coloneqq \Norm{p (\cdot, T)- \pht (\cdot, T)}{\L^2(\Omega(T))} \\
        \mathrm{velERR}_{h,\tau} & \coloneqq \Big[ \Norm{\u (\cdot, T)- \uht (\cdot, T)}{\L^2(\Omega(T))}^2 
         + \nu \sum_{n = 1}^N \int_{\In} \Norm{\u (\cdot, t)- \uht (\cdot, t)}{1,h,t}^2 \dt \Big]^{1/2}.
    \end{aligned}
\end{equation*}
\subsection{Test 1: convergence rates}\label{sec:test1_conv_rates}
In this section, we assess the convergence rates predicted by Corollary~\ref{cor:u-minus-uh.error.est}. Specifically, in Section~\ref{sec:test1_conv_rates.space}, we verify the spatial convergence rates by prescribing suitable manufactured solutions $(\u, p)$ such that the temporal error is negligible. Conversely, in Section~\ref{sec:test1_conv_rates.time}, we verify the temporal convergence rates by ensuring a negligible spatial error. Finally, in Section~\ref{sec:test1_conv_rates.space-time}, we perform a comprehensive space--time convergence test. In all these tests, we set the viscosity parameter~$\nu = 1$.
\subsubsection{Convergence in space}\label{sec:test1_conv_rates.space}
We prescribe the exact solution $(\u,p)$ to problem~\eqref{eq:model-problem-ALE} as
\begin{equation*}
    \u (\bx,t) =
    \begin{pmatrix} 
        \cos(x_2) + t \left( \frac{x_1}{1 + t x_2} \sin(\frac{x_1}{1 + t x_2}) + \cos(\frac{x_1}{1 + t x_2}) \right) \\ 
        \sin(\frac{x_1}{1 + t x_2}) 
    \end{pmatrix} 
    \quad \text{and} \quad p (\bx,t) = t \cos(\pi x_1) \cos(\pi x_2) .
\end{equation*}
The forcing term $\f$ and the Dirichlet boundary conditions are derived accordingly.

Note that, when mapped to the reference configuration, the velocity $\hu$ is a polynomial of degree $1$ with respect to time. Conversely, due to the nonlinear transformation of the coordinates inside the arguments of the trigonometric functions, the mapped pressure $\widetilde{p}$ is not a polynomial in time. Nevertheless, thanks to the pressure-robustess property of the proposed formulation, which will be analyzed in detail in Section~\ref{sec:test3_pressure-robustness} below, the temporal discretization error stemming from the pressure component does not pollute the velocity error.
Consequently, by setting the degree of approximation in time to $\ell = 1$, the temporal component of the velocity error is negligible compared to the spatial error. We fix the time step to $\tau = 0.0625$ and consider a sequence of spatial meshes for $\Omega_0$ with mesh sizes~$h \approx 0.149582, \, 0.0793161,\,  0.0325956,\,  0.01227$. As shown in Figure~\ref{fig:convInSpace_convInTime}, the error $\mathrm{velERR}_{h,\tau}$ decays at the expected rate of $\mathcal{O}(h)$.
\subsubsection{Convergence in time}\label{sec:test1_conv_rates.time}
We now prescribe the exact solution $(\u,p)$ to problem~\eqref{eq:model-problem-ALE} as
\begin{equation*}
    \u (\bx,t) = \sin(2\pi t)
    \begin{pmatrix}
        t x_2 + t^2 \\
        t
    \end{pmatrix}
    \quad \text{and} \quad p (\bx,t) = 0.
\end{equation*}
The forcing term $\f$ and the Dirichlet boundary conditions are obtained accordingly. 

In this case, the corresponding~$\hu$ is a polynomial of degree $1$ with respect to the spatial variables. Consequently, since the spatial polynomial degree is fixed to~$k=1$, the spatial components are represented exactly by our method, making the spatial error negligible compared to the temporal one. We fix the mesh size to $h = 0.149582$ and consider a sequence of systematically halved time steps $\tau \in \{ 0.5, 0.25, 0.125, 0.0625, 0.03125, 0.015625 \}$ over the time interval $(0,1)$. As shown in Figure~\ref{fig:convInSpace_convInTime}, the error $\mathrm{velERR}_{h,\tau}$ evaluated for different values of the time-DG degree $\ell \in \{0,\, 1,\, 2\}$ decays at the expected rate of $\mathcal{O}(\tau^{\ell+1})$.
\begin{figure}[ht]
    \centering
    \begin{subfigure}[b]{0.45\textwidth}
        \centering
        \begin{tikzpicture}
            \begin{loglogaxis}[
                    width=\textwidth,     
                    height=0.75\textwidth,   
                    xlabel={$h$},
                    xmin=0.01, xmax=0.2,
                    ymin=1e-4, ymax=1e-1, 
                    grid=both,
                    major grid style={solid, line width=0.4pt, draw=black!30},
                    minor grid style={solid, line width=0.2pt, draw=black!10},
                    legend pos=south east,
                    legend style={nodes={scale=0.75, transform shape}},
                    ]
                \addplot[blue, mark=*, line width=1.5] table[x=h,y=errTotVelSC_degree_1] {NumericalResults/errTotVelSC.txt};
                \addlegendentry{$k=1$}
                \addplot[blue, dashed, domain=0.012:0.15, line width=1.0] {0.03*x};
                \addlegendentry{$\mathcal{O}(h)$}
            \end{loglogaxis}
        \end{tikzpicture}
    \end{subfigure}
    \hfill
    \begin{subfigure}[b]{0.45\textwidth}
        \centering
        \begin{tikzpicture}
            \begin{loglogaxis}[
                    width=\textwidth,     
                    height=0.75\textwidth,   
                    xlabel={$\tau$},
                    xmin=1e-2, xmax=1,
                    ymin=1e-5, ymax=10, 
                    grid=both,
                    major grid style={solid, line width=0.4pt, draw=black!30},
                    minor grid style={solid, line width=0.2pt, draw=black!10},
                    legend pos=south east,
                    legend style={nodes={scale=0.75, transform shape}},
                    ]
                \addplot[black, mark=*, line width=1.5] table[x=tau,y=errTotVelTC_degree_0] {NumericalResults/errTotVelTC.txt};
                \addlegendentry{$\ell=0$}
                \addplot[black, dashed, domain=0.015625:0.5, line width=1.0] {12*x};
                \addlegendentry{$\mathcal{O}(\tau)$}
                \addplot[blue, mark=square, line width=1.5] table[x=tau,y=errTotVelTC_degree_1] {NumericalResults/errTotVelTC.txt};
                \addlegendentry{$\ell=1$}
                \addplot[blue, dashed, domain=0.015625:0.5, line width=1.0] {10*x^2};
                \addlegendentry{$\mathcal{O}(\tau^2)$}
                \addplot[red, mark=triangle, line width=1.5] table[x=tau,y=errTotVelTC_degree_2] {NumericalResults/errTotVelTC.txt};
                \addlegendentry{$\ell=2$}
                \addplot[red, dashed, domain=0.015625:0.5, line width=1.0] {5*x^3};
                \addlegendentry{$\mathcal{O}(\tau^3)$}
            \end{loglogaxis}
        \end{tikzpicture}
    \end{subfigure}
    \caption{Test 1: Convergence plots for the velocity error $\mathrm{velERR}_{h,\tau}$ in space (left panel) and time (right panel).}  \label{fig:convInSpace_convInTime}
\end{figure}
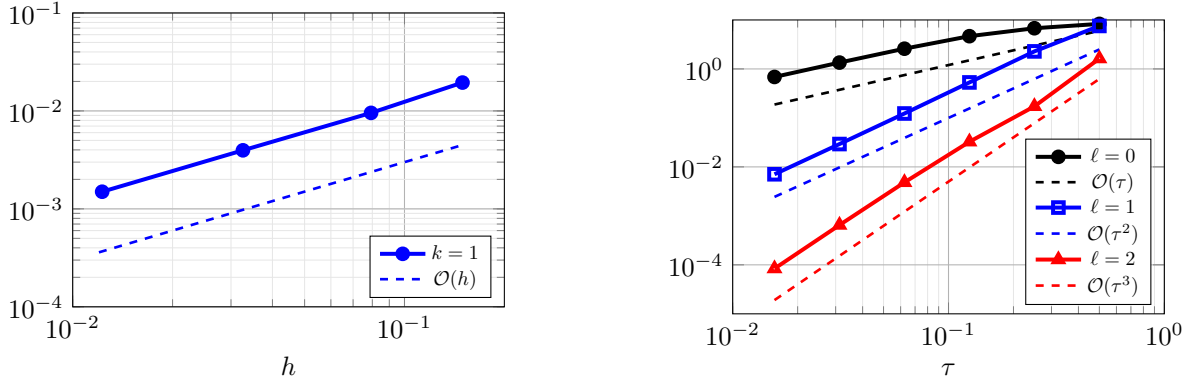
\subsubsection{Space--time convergence}\label{sec:test1_conv_rates.space-time}
In this test, we prescribe the smooth exact solution $(\u,p)$ to problem~\eqref{eq:model-problem-ALE} as
\begin{equation*}
    \u(\bx, t) = \pi \cos(\pi t) 
    \begin{pmatrix} 
        \sin(\pi x_1) \cos(\pi x_2) \\ 
        -\cos(\pi x_1) \sin(\pi x_2) 
    \end{pmatrix}
    \quad \text{and} \quad
    p(\bx, t) = \cos(\pi t) \cos(\pi x_1) \cos(\pi x_2),
\end{equation*}
with the corresponding forcing term~$\f$ and Dirichlet boundary conditions.

In this test, we verify the simultaneous space--time convergence rates by matching the discretization parameters to equilibrate the spatial and temporal errors. Specifically, we evaluate the error $\mathrm{velERR}_{h,\tau}$ by coupling the mesh size and time step according to $\tau \simeq h^{\frac{1}{\ell+1}}$, whereas for the final-time errors $\mathrm{velERR}_{T}$ and~$\mathrm{preERR}_{T}$, we impose the relation $\tau \simeq h^{\frac{2}{\ell+1}}$. The tests are performed on the sequence of meshes with sizes $h \approx 0.149582,\,  0.0793161,\, 0.0325956,\,  0.01227$.

As shown in Figure~\ref{fig:space_time_convergence_tot_and_final_time}, the error $\mathrm{velERR}_{h,\tau}$ decays at the expected rate of $\mathcal{O}(h)$. The velocity error at the final time step, $\mathrm{velERR}_{T}$, exhibits an $\mathcal{O}(h^2)$ reduction. Regarding the pressure error $\mathrm{preERR}_{T}$, although the present theoretical analysis does not explicitly cover its convergence bounds, we empirically observe a decay rate of $\mathcal{O}(h)$.
\begin{figure}[ht]
    \centering
    \begin{subfigure}[b]{0.45\textwidth}
        \centering
        \begin{tikzpicture}
            \begin{loglogaxis}[
                    width=\textwidth,     
                    height=0.75\textwidth,   
                    xlabel={$h$},
                    xmin=0.01, xmax=0.2,
                    ymin=0.01, ymax=0.5, 
                    grid=both,
                    major grid style={solid, line width=0.4pt, draw=black!30},
                    minor grid style={solid, line width=0.2pt, draw=black!10},
                    legend pos=south east,
                    legend style={nodes={scale=0.75, transform shape}},
                    ]
                \addplot[black, mark=*, line width=1.5] table[x=h,y=errTotVel_degree_0] {NumericalResults/errTotVel.txt};
                \addlegendentry{$\ell=0$}
                \addplot[blue, mark=square, line width=1.5] table[x=h,y=errTotVel_degree_1] {NumericalResults/errTotVel.txt};
                \addlegendentry{$\ell=1$}\addplot[red, mark=triangle, line width=1.5] table[x=h,y=errTotVel_degree_2] {NumericalResults/errTotVel.txt};
                \addlegendentry{$\ell=2$}
                \addplot[black, dashed, domain=0.012:0.15, line width=1.0] {x};
                \addlegendentry{$\mathcal{O}(h)$}
            \end{loglogaxis}
        \end{tikzpicture}
    \end{subfigure}
    \hfill
    \begin{subfigure}[b]{0.45\textwidth}
        \centering
        \begin{tikzpicture}
            \begin{loglogaxis}[
                    width=\textwidth,     
                    height=0.75\textwidth,   
                    xlabel={$h$},
                    xmin=0.01, xmax=0.2,
                    ymin=1e-5, ymax=10, 
                    grid=both,
                    major grid style={solid, line width=0.4pt, draw=black!30},
                    minor grid style={solid, line width=0.2pt, draw=black!10},
                    legend pos=south east,
                    legend style={nodes={scale=0.5, transform shape}},
                    ]
                \addplot[black, mark=*, line width=1.5] table[x=h,y=errFinalTimeL2Vel_degree_0] {NumericalResults/errFinalTimeL2Vel.txt};
                \addlegendentry{$\ell=0$, $\u$}
                \addplot[blue, mark=*, line width=1.5] table[x=h,y=errFinalTimeL2Vel_degree_1] {NumericalResults/errFinalTimeL2Vel.txt};
                \addlegendentry{$\ell=1$, $\u$}
                \addplot[red, mark=*, line width=1.5] table[x=h,y=errFinalTimeL2Vel_degree_2] {NumericalResults/errFinalTimeL2Vel.txt};
                \addlegendentry{$\ell=2$, $\u$}
                \addplot[black, domain=0.012:0.15, line width=1.0] {3*x^2};
                \addlegendentry{$\mathcal{O}(h^2)$}
                \addplot[black, mark=square*, line width=1.5] table[x=h,y=errFinalTimeL2Pre_degree_0] {NumericalResults/errFinalTimeL2Pre.txt};
                \addlegendentry{$\ell=0$, $p$}
                \addplot[blue, mark=square*, line width=1.5] table[x=h,y=errFinalTimeL2Pre_degree_1] {NumericalResults/errFinalTimeL2Pre.txt};
                \addlegendentry{$\ell=1$, $p$}
                \addplot[red, mark=square*, line width=1.5] table[x=h,y=errFinalTimeL2Pre_degree_2] {NumericalResults/errFinalTimeL2Pre.txt};
                \addlegendentry{$\ell=2$, $p$}
                \addplot[black, dashed, domain=0.012:0.15, line width=1.0] {20*x};
                \addlegendentry{$\mathcal{O}(h)$}
            \end{loglogaxis}
        \end{tikzpicture}
    \end{subfigure}
    \caption{Test 1: Space--time convergence analysis under simultaneous mesh refinement and time-step reduction for different values of the time-DG degree $\ell \in \{ 0, 1, 2\}$. Error $\mathrm{velERR}_{h,\tau}$ with $\tau \simeq h^{\frac{1}{\ell+1}}$ (left panel) and errors $(\mathrm{velERR}_{T},\mathrm{preERR}_{T})$ with $\tau \simeq h^{\frac{2}{\ell+1}}$ (right panel).}
    \label{fig:space_time_convergence_tot_and_final_time}
\end{figure}
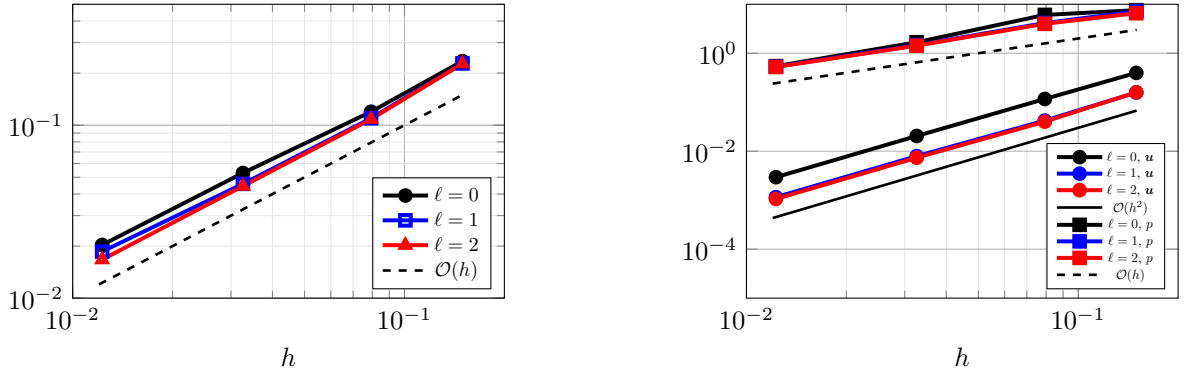
\subsection{Test 2: viscosity-robustness}\label{sec:test2_nu-robustness}
In this section, we investigate the robustness of the proposed method with respect to the fluid viscosity~$\nu$. For this analysis, we adopt the same smooth exact solution $(\u,p)$ used in the space--time convergence test of Section~\ref{sec:test1_conv_rates.space-time}. We fix the spatial mesh size to $h=0.149582$ and evaluate the method's behavior for $\nu \in \{1, 10^{-1}, 10^{-2}, \dots, 10^{-10}\}$.
The numerical results, reported in Figure~\ref{fig:nuRobustness}, show that the error $\mathrm{velERR}_{h,\tau}$ has an initial small increase and then stabilizes, remaining essentially constant for smaller values of $\nu$ for all tested time-DG degrees $\ell \in \{ 0, 1, 2, 3 \}$. Notice that, while our current theoretical framework guarantees $\nu$-robustness strictly for the lowest-order cases $\ell \in \{ 0,1 \}$, the numerical evidence strongly indicates that this desirable property extends to higher-order temporal approximations (with~$\ell \geq 2$) as well (see also Remark \ref{rem:nu-robustenss-high-degree}).
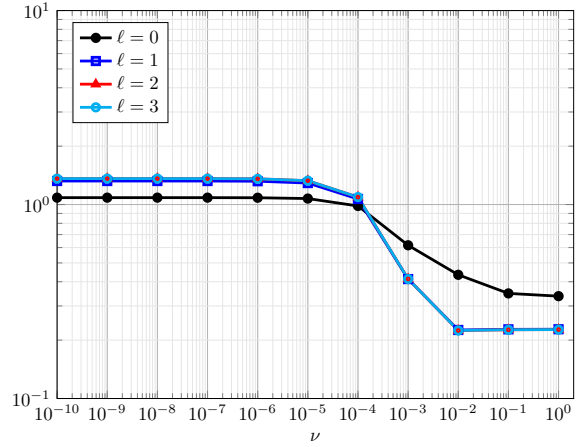
\begin{figure}[ht]
    \centering
    \begin{tikzpicture}[scale = 0.7]
        \begin{loglogaxis}[
                width=0.7\textwidth,     
                height=0.55\textwidth,   
                xlabel={$\nu$},
                xmin=1e-10, xmax=2,
                ymin=1e-1, ymax=10, 
                grid=both,
                major grid style={solid, line width=0.4pt, draw=black!30},
                minor grid style={solid, line width=0.2pt, draw=black!10},
                legend pos=north west,
                ]
            \addplot[black, mark=*, line width=1.5] table[x=nu,y=errTotVelNu_degree_0] {NumericalResults/errTotVelNu.txt};
            \addlegendentry{$\ell=0$}
            \addplot[blue, mark=square, line width=1.5] table[x=nu,y=errTotVelNu_degree_1] {NumericalResults/errTotVelNu.txt};
            \addlegendentry{$\ell=1$}
            \addplot[red, mark=triangle, line width=1.5] table[x=nu,y=errTotVelNu_degree_2] {NumericalResults/errTotVelNu.txt};
            \addlegendentry{$\ell=2$}
            \addplot[cyan, mark=o, line width=1.5] table[x=nu,y=errTotVelNu_degree_3] {NumericalResults/errTotVelNu.txt};
            \addlegendentry{$\ell=3$}
        \end{loglogaxis}
    \end{tikzpicture}
    \caption{Test 2: $\nu$-robustness of the error $\mathrm{velERR}_{h,\tau}$ for different values of the time-DG degree $\ell \in \{ 0, 1, 2\}$.}
    \label{fig:nuRobustness}
\end{figure}
\subsection{Test 3: pressure-robustness}\label{sec:test3_pressure-robustness}
This section is devoted to verifying the pressure-robustness of the proposed numerical scheme by evaluating whether the velocity error depends on the pressure field and its discrete approximation. To this end, we prescribe the exact velocity field as
\begin{equation*}
    \u(\bx, t) = 
    \begin{pmatrix} 
        x_2 + \frac{t x_1^2}{(1 + t x_2)^3} \\
        \frac{x_1}{(1 + t x_2)^2}
    \end{pmatrix}.
\end{equation*}
By construction, the mapped velocity field satisfies $\hu \in \Vhk \otimes \mathbb{P}^0([0,1])$, therefore the proposed method should theoretically reproduce the exact velocity up to machine precision. Conversely, the exact pressure field is chosen as in the space--time convergence test (cf. Section~\ref{sec:test1_conv_rates.space-time}), which cannot be represented exactly by the discrete pressure space, thereby generating a nontrivial pressure error.

The results of this pressure-robustness test are presented in Table~\ref{tab:pre-robustness}. As expected, for both $\ell = 0$ and $\ell = 1$, the velocity error $\mathrm{velERR}_{h,\tau}$ is machine zero for all considered mesh sizes $h$, while the pressure error $\mathrm{preERR}_{T}$ is large. This contrast empirically confirms that the velocity error is entirely decoupled from the pressure approximation, in accordance with the pressure-robustness of the proposed formulation.
\begin{table}[htpb]
    \centering
    \begin{tabular}{lllll}
    \toprule
    & \multicolumn{2}{c}{$\ell = 0$} & \multicolumn{2}{c}{$\ell = 1$}\\
    \cmidrule(lr){2-3} \cmidrule(lr){4-5} 
    $h$ & $\mathrm{velERR}_{h,\tau}$ & $\mathrm{preERR}_{T}$ & $\mathrm{velERR}_{h,\tau}$ & $\mathrm{preERR}_{T}$ \\
    \midrule
    $0.1496$ & $4.07 \cdot 10^{-11}$ & $1.24 \cdot 10^{-1}$ & $4.20 \cdot 10^{-11}$ & $1.10 \cdot 10^{-1}$ \\
    $0.0793$ & $4.05\cdot 10^{-13}$ & $6.76 \cdot 10^{-2}$ & $3.48 \cdot 10^{-13}$ & $5.89 \cdot 10^{-2}$ \\    
    $0.0326$ & $1.51 \cdot 10^{-12}$ & $3.11 \cdot 10^{-2}$ & $7.43 \cdot 10^{-12}$ & $2.77 \cdot 10^{-2}$ \\
    \bottomrule 
     \end{tabular}
    \caption{Test 3: pressure-robustness test for temporal degrees $\ell \in \{0, 1\}$. The velocity error is zero up to machine precision, confirming that the velocity approximation is entirely unaffected by the pressure discretization error.}
    \label{tab:pre-robustness}
\end{table}

\subsection{Test 4: an ovalization benchmark}\label{sec:test4_ovalization}
In this final numerical test, we investigate a more physically challenging problem where the fluid motion is entirely driven by the deformation of the domain. 
We consider an initial circular domain $\Omega_0 = \{ \by \in \mathbb{R}^2 : |\by| < 1 \}$ at $t=0$, which is progressively squeezed into an ellipse through the area-preserving ALE map $\A^{\mathrm{D}}$ (cf. \eqref{eq:ALEMaps-tests}).

We solve the ALE-Stokes problem~\eqref{eq:model-problem-ALE} with zero external forcing ($\boldsymbol{f} = \boldsymbol{0}$) and a kinematic viscosity $\nu = 1$. The system is driven exclusively by the nonhomogeneous Dirichlet boundary conditions assigned on the moving boundary~$\partial\Omega(t)$, where we enforce the fluid velocity to match the domain velocity~$\u = \bw$. This test serves as a benchmark to verify whether the proposed formulation can accurately simulate fluid dynamics purely induced by the squeezing and stretching of the domain boundaries.

The spatial domain is discretized using a triangular mesh with $h \approx 0.1463$ (approximating the initial circular boundary with 64 points). For the temporal discretization, we set the final time to $T=1$, the time step to $\tau = 1/6$, and the time-DG degree to $\ell = 1$. 

Figure~\ref{fig:disc_vel} shows the discrete pressure $\pht$ and the velocity field $\uht$ at three different time steps: $t \in \{0, 0.5, 1\}$. As expected from the physical setup, the pressure distribution directly reflects the domain's deformation. We have negative pressure zones at the lateral boundaries where the mesh stretches outward, creating a suction effect. Conversely, we have positive pressure regions at the top and bottom boundaries where the domain is actively squeezed inward. 
The velocity vectors clearly follow the kinematic deformation of the mesh: the fluid flows inward along the vertical axis and is expelled outward along the horizontal axis.

Furthermore, to quantitatively assess the exact mass conservation of the proposed scheme under continuous domain deformation, we monitor the divergence of the discrete velocity. As reported in Table~\ref{tab:test4_divergence}, the $L^2$ norm of the divergence remains at the level of machine precision at all observed times, confirming that the divergence-free constraint~$\divx \uht = 0$ is strictly satisfied.
\begin{table}[htpb]
    \centering
    \begin{tabular}{cc}
        \toprule
        Time $t$ & $\| \divx \uht \|_{L^2(\Omega(t))}$ \\
        \midrule
        $0.0$ & $2.47 \cdot 10^{-14}$ \\
        $0.5$ & $1.22 \cdot 10^{-14}$ \\
        $1.0$ & $2.45 \cdot 10^{-14}$ \\
        \bottomrule
    \end{tabular}
    \caption{Test 4: $L^2$ norm of the discrete velocity divergence evaluated at different times.}
    \label{tab:test4_divergence}
\end{table}
\begin{figure}[ht]
    \centering
    \begin{subfigure}[b]{0.6\textwidth}
        \centering
        \includegraphics[width=\linewidth]{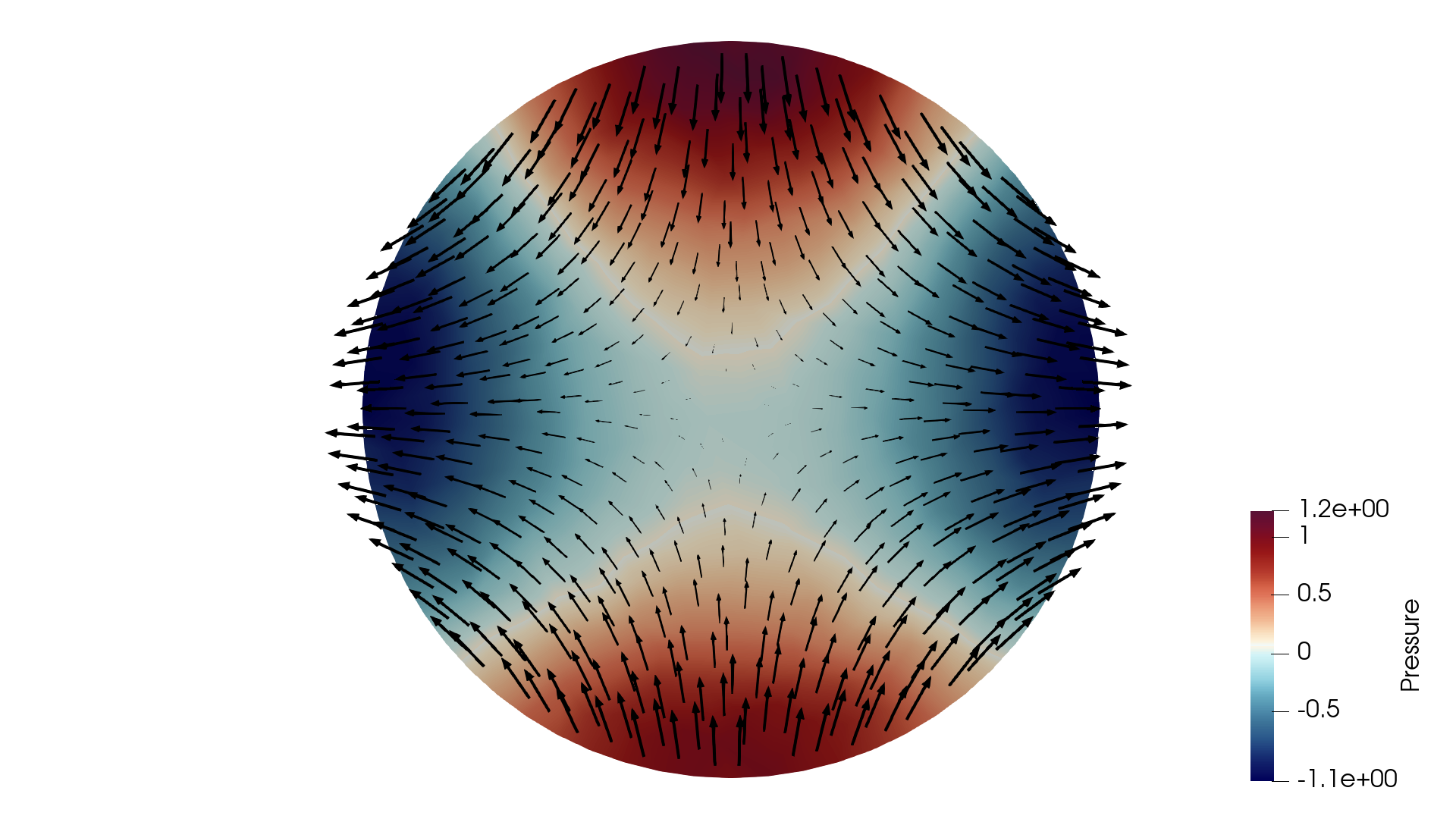}
        \caption{$t=0$}
    \end{subfigure}\\
    \centering
    \begin{subfigure}[b]{0.4\textwidth}
        \centering
        \includegraphics[width=\linewidth]{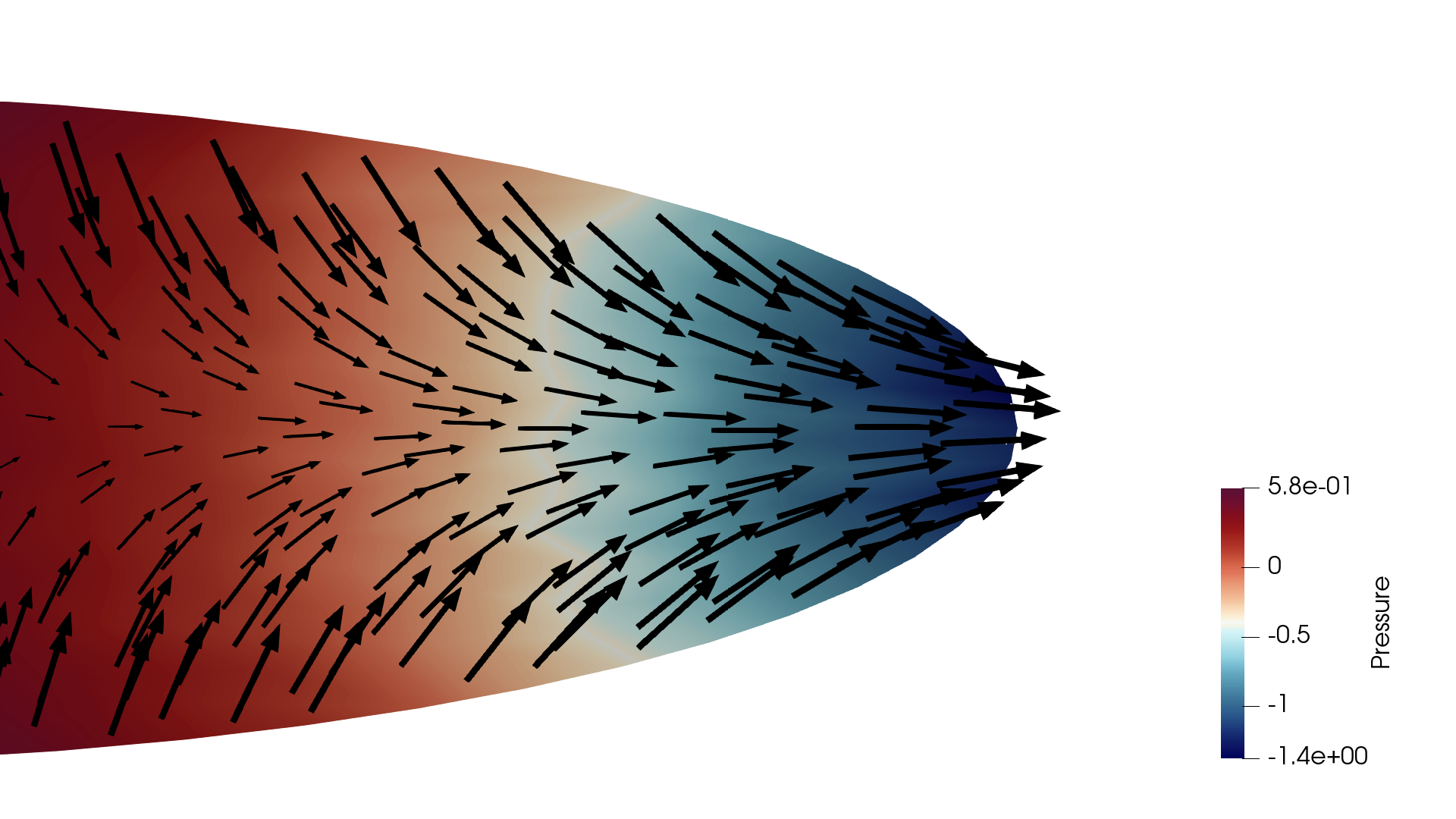}
        \caption{$t=0.5$}
    \end{subfigure}
    \hfill
    \begin{subfigure}[b]{0.4\textwidth}
        \centering
        \includegraphics[width=\linewidth]{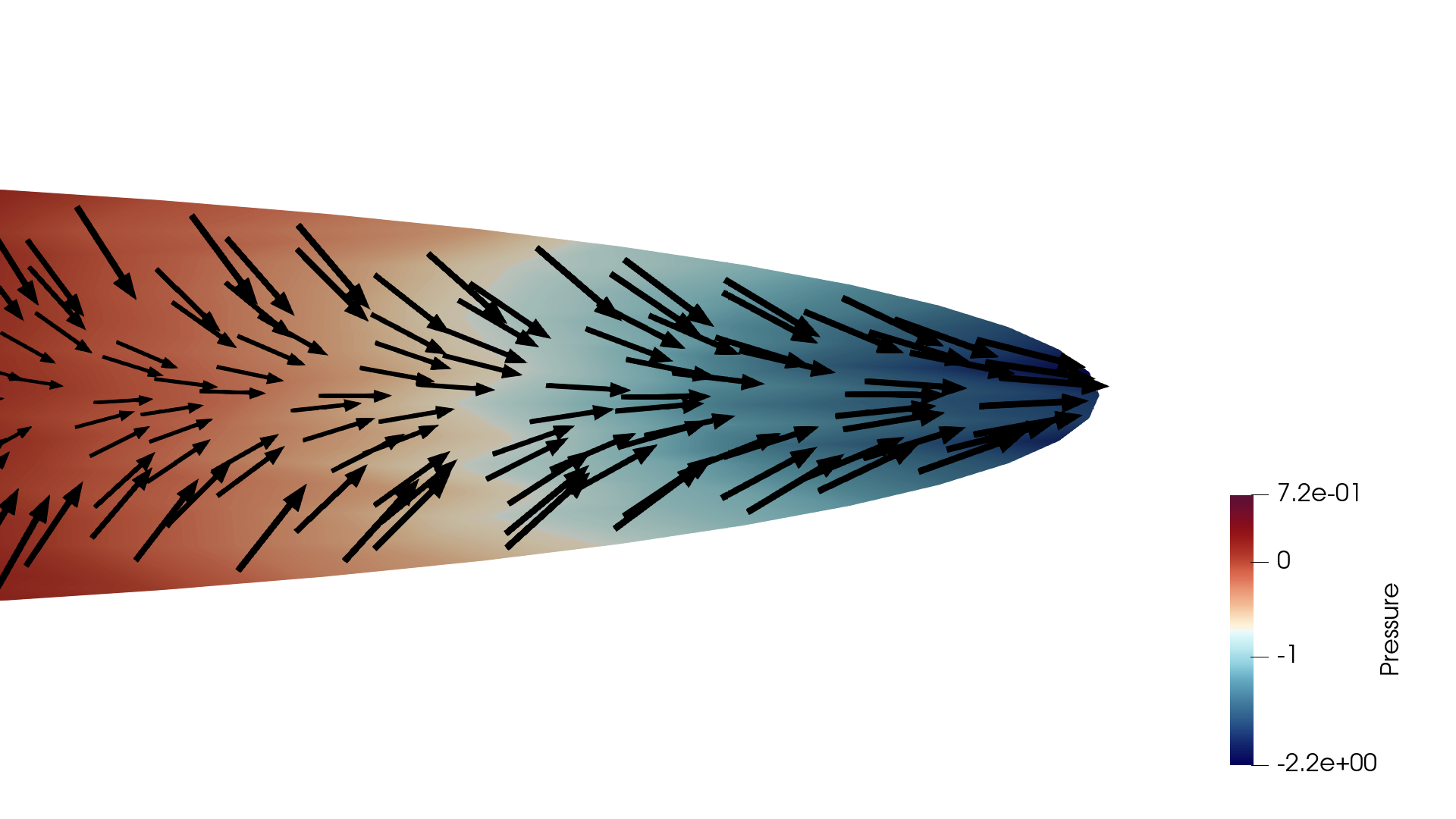}
        \caption{$t=1$}
    \end{subfigure}
    \caption{Test 4: discrete velocity $\uht$ and discrete pressure $\pht$ for different times $t \in \{0, 0.5, 1\}$.}
    \label{fig:disc_vel}
\end{figure}
\section*{Acknowledgements}
All authors were partially supported by the European Union (ERC Synergy, NEMESIS, project number
101115663). Views and opinions expressed are however those of the authors only and do not necessarily
reflect those of the EU or the ERC Executive Agency. All authors are also member of the INdAM-GNCS
group.

\bibliographystyle{plain}
\bibliography{references}
\appendix
\section{Appendix: Proof of~Lemma~\ref{lemma:equivalence} }\label{app:bounds}
In this section, we detail the hidden constants in Lemma~\ref{lemma:equivalence}. We show the proof only for the third bound, which requires careful handling of the Piola map~$\phi_t$, as the remaining ones can be shown with analogous arguments.
We first define the following constants:
\begin{alignat*}{2}
    J^{\star} &\coloneqq \esssup_{(\by, t) \in \Omega_0 \times [0, T]} |\det \Nablay \At(\by)|, &\qquad
    J_{\star} &\coloneqq \essinf_{(\by, t) \in \Omega_0 \times [0, T]} |\det \Nablay \At(\by)|, \\ 
    J_{\sup} &\coloneqq \esssup _{(\by, t) \in \Omega_0 \times [0, T]} \Norm{\Nablay \At(\by)}{\infty}, &\qquad
    \tilde{J}_{\sup} &\coloneqq \esssup _{(\by, t) \in \Omega_0 \times [0, T]} \Norm{\Nablay^2 \At(\by)}{\infty}.
\end{alignat*}
\begin{lemma}\label{lemma:equivalence.detailed}
    Let Assumption~\ref{asm:At} on the map~$\A$ hold.
    Let $\hv \in \boldsymbol {H}^2(\ThO)$  and $\widetilde{q} \in L^2(\Omega_0)$.
    For any $K \in \ThO$, any $F \in \Fh$, and any $t \in [0,T]$, the following inequalities hold:
    \begin{subequations}
        \begin{alignat}{3}
            \label{eq:trace.NEW.detailed}
            d^{-\frac{3}{2}} \frac{C_{\A}^2}{J^{\star}}  \Norm{\hv}{\L^2(\partial K)}^2 &\leq \Norm{\phi_t\hv}{\L^2(\partial \Kt)}^2 \leq d^{\frac{3}{2}} \frac{J_{\sup}^2}{J_{\star}}  \Norm{\hv}{\L^2(\partial K)}^2 , \\
            \label{eq:trace-grad.NEW.detailed}
            \Norm{\Nablax (\phi_t \hv)}{\L^2(\partial \Kt)}^2  & \leq d^{\frac{1}{2}} J_{\sup} \mathfrak{C}_1 C_{\mathrm{c.t.}}  \left( \hK^{-1} \Norm{\hv}{\L^2(K)}^2 + (\hK + \hK^{-1})  \Norm{\Nablay \hv}{\L^2(K)}^2 + \hK \Norm{\Nablay^2 \hv}{\L^2(K)}^2\right), \\
            \label{eq:equiv-vh.1.detailed}
            \Norm{\Nablax (\phi_t \hv)}{\L^2(\Kt)}^2 & \le \mathfrak{C}_1 \left(  \Norm{\hv}{\L^2(K)}^2 +  \Norm{\Nablay \hv}{\L^2(K)}^2  \right) , 
            \\
            \label{eq:equiv-vh.2.detailed}
            \Norm{\Nablay \hv}{\L^2(K)}^2 & \le \mathfrak{C}_2 \left(\Norm{\phi_t \hv}{\L^2(\Kt)}^2 + \Norm{\Nablax (\phi_t \hv)}{\L^2(\Kt)}^2 \right)  , 
            \\
            \label{eq:equiv-vh-2.detailed}
            d^{-\frac{3}{2}} \frac{c_{\A}^2}{J^{\star} J_{\sup}}
            \Norm{\jump{\hv}_{F}}{\L^2(F)}^2
            &\le \Norm{\jump{\phi_t \hv}_{\Ft}}{\L^2(\Ft)}^2 \le
            d^{\frac{3}{2}} \frac{J_{\sup}^2}{c_{\A} J_{\star}}
             \Norm{\jump{\hv}_{F}}{\L^2(F)}^2, 
            \\
            \label{eq:equiv-vh-3.detailed}
            d^{-1} \frac{c_{\A}^2}{J^{\star}}
            \Norm{\hv}{\L^2(\Omega_0)}^2
            &\le \Norm{\phi_t \hv}{\L^2(\Omega(t))}^2 \le
            d \frac{J_{\sup}^2}{J_{\star}}
            \Norm{\hv}{\L^2(\Omega_0)}^2,
            \\
            \label{eq:equiv-qht.detailed}
            \frac{1}{J^{\star}} \Norm{\widetilde{q} \circ \At^{-1}}{L^2(\Omega(t))}^2  & \le \Norm{\widetilde{q}}{L^2(\Omega_0)}^2 \le \frac{1}{J_{\star}}  \Norm{\widetilde{q} \circ \At^{-1}}{L^2(\Omega(t))}^2,
        \end{alignat}
    \end{subequations}
    where
    \begin{align*}
        \mathfrak{C}_1 &\coloneqq 3d^2 \frac{1}{c_{\A}^2} J_{\star}^{-1} \max \left( \tilde{J}_{\sup}^2 \left( d^3 \, (d!)^2 J_{\star}^{-2} J_{\sup}^{2d} +1\right) , J_{\sup}^2 \right) \\
        \mathfrak{C}_2 &\coloneqq 3 d^2  \max \left( \tilde{J}_{\sup}^2 \left(d^2 (d!)^2 J_{\star}^{-1}J_{\sup}^{2d} + J^{\star} c_{\A}^{-4} \right) , J_{\sup}^2 J^{\star} c_{\A}^{-2} \right).
    \end{align*}
    Moreover, if $\hv \in \Vhk$, then the following discrete trace estimates hold:
    \begin{subequations}
        \begin{alignat*}{3}
            \Norm{\phi_t\hv}{\L^2(\partial \Kt)}^2   &\leq d \frac{J^{\star} C_{\mathrm{d.t.}}}{c_{\A}^2} \hK^{-1}   \Norm{\phi_t\hv}{\L^2(\Kt)}^2, \\
            \Norm{\Nablax (\phi_t\hv)}{\L^2(\partial \Kt)}^2  &\leq d^{\frac{1}{2}} \frac{C_{\mathrm{d.t.}}}{c_{\A}} \hK^{-1}  \Norm{\Nablax (\phi_t\hv)}{\L^2(\Kt)}^2,
        \end{alignat*}
    \end{subequations}
    where $C_{\mathrm{c.t.}}$ and $C_{\mathrm{d.t.}}$ denote the continuous and discrete trace inequality constants, respectively.
\end{lemma}
\begin{proof}[Proof of bound~\eqref{eq:equiv-vh.1.detailed}]
    We start by noticing that, from the definition of the Piola map $\phi_t$ and applying the chain rule, we have
    \[
        \Nablax (\phi_t \hv_h) (\bx) = \Big[((\mathfrak{M}_1 + \mathfrak{M}_2)  \circ \At^{-1}) \Nablax (\At^{-1})\Big] (\bx),
    \]
    with
    \[ 
        \mathfrak{M}_1 \coloneqq \Nablay ( (\det \Nablay \At)^{-1}) \otimes \left[\Nablay \At \hv_h \right] \quad \text{and} \quad 
        \mathfrak{M}_2 \coloneqq (\det \Nablay \At)^{-1} ( \Nablay^2 \At \star \hv_h+\Nablay \At \Nablay \hv_h).
    \]
    Above, $\Nablay ( (\det \Nablay \At)^{-1}) \otimes \left[\Nablay \At \hv_h \right]$ denotes the tensor product between the vectors $\Nablay ( (\det \Nablay \At)^{-1})$ and $\Nablay \At \hv_h$, while $\Nablay^2 \At \star \hv_h$ denotes the tensor contraction along the third index between the third-order tensor $\Nablay^2 \At $ and the vector $\hv_h$.
    A change of variable  leads to
    \begin{equation*}
        \Norm{\Nablax (\phi_t \hv_h)}{\L^2(\Kt)}^2 =
        \int_{K}
        \Norm{[(\mathfrak{M}_1 + \mathfrak{M}_2) (\by)] [ \Nablax (\At^{-1}) (\At (\by))]}{}^2 \left| \det \Nablay \At (\by) \right| \dy
    \end{equation*}
    Applying the triangle and the Cauchy--Schwarz inequalities, along with the fact that (recalling Assumption~\ref{asm:At}\ref{asm:At-2})
    $$
        \Norm{\Nablax (\At^{-1}) (\At (\by))}{} = \Norm{[\Nablay \At (\by)]^{-1}}{} \leq d^{1/2}\Norm{[\Nablay \At (\by)]^{-1}}{\infty} \leq  \frac{d^{1/2}}{c_{\A}} \, ,
    $$
    we obtain
    \begin{equation}\label{eq:grad.piola.map.proof.init}
        \Norm{\Nablax (\phi_t \hv_h)}{\L^2(\Kt)}^2 
        \leq
        3d \frac{1}{c_{\A}^2} \int_{K} (\mathfrak{T}_{t,1} + \mathfrak{T}_{t,2})(\by) \dy,  
    \end{equation}
    with 
    $\mathfrak{T}_{t,1} (\by)  \coloneqq  \Norm{\mathfrak{M}_1 (\by)}{}^2 \left| \det \Nablay \At (\by) \right|$ and $\mathfrak{T}_{t,2} (\by)  \coloneqq \Norm{\mathfrak{M}_2 (\by)}{}^2 \left| \det \Nablay \At (\by) \right|$.
    
    The second term $\mathfrak{T}_{t,2} (\by)$ is easily bounded using the constants introduced at the beginning of the section
    \begin{equation}\label{eq:grad.piola.map.proof.term2.3}
        \mathfrak{T}_{t,2} (\by) \leq d J_{\star}^{-1} \tilde{J}_{\sup}^2 \left( | \hv_h (\by) |^{2}  + \Norm{\Nablay \hv_h (\by)}{}^2\right),
    \end{equation}
    where the constant $d$ comes from the equivalence between the Euclidean and the $\ell^{\infty}$ norm.
    
    For the first term $\mathfrak{T}_{t,1} (\by)$, using the tensor product property, basic derivation rules, and the Cauchy--Schwarz inequality, we write
    \begin{align*}
        \mathfrak{T}_{t,1} (\by) &\leq 
        \left| \det \Nablay \At (\by) \right|^{-3} \Norm{\Nablay \At (\by) }{}^2 |\Nablay (\det \Nablay \At (\by))|^2 | \hv_h (\by) |^{2}\\ 
        &\leq
        d^2 J_{\star}^{-3} J_{\sup}^2 |\Nablay (\det \Nablay \At (\by))|_{\ell^{\infty}}^2 | \hv_h (\by) |^2.
    \end{align*}
    Observing that, for any $j \in \{ 1,\ldots,d \}$,
    \begin{align*}
        | \partial_{j} (\det \Nablay \At (\by)) |
        &=  \left| \sum_{\sigma \in \mathcal{S}_d} \mathrm{sign} (\sigma) \partial_{j} \left(\prod_{i=1}^d \partial_{\sigma(i)} ( (\At (\by) )_i  ) \right) \right|   \\ 
        &= \left| \sum_{\sigma \in \mathcal{S}_d} \mathrm{sign} (\sigma) \left[ \sum_{i=1}^d
        \partial_{j} \partial_{\sigma(i)} ( (\At (\by) )_i  )  \prod_{\ell\neq i}  \partial_{\sigma(\ell)} ( (\At (\by) )_{\ell}  ) 
         \right] \right| \\
         &\leq 
          \sum_{\sigma \in \mathcal{S}_d} \left[ \sum_{i=1}^d \big| [\Nablay^2 \At (\by)  ]_{i,\sigma(i),j} \big|  \prod_{\ell\neq i}  
          \big| [\Nablay \At (\by)]_{\ell,\sigma(\ell)}  \big|
         \right]  \\ 
         &\leq 
         d \, d! \,  J_{\sup}^{d-1} \tilde{J}_{\sup},
    \end{align*}
    we conclude that
    \begin{equation}\label{eq:grad.piola.map.proof.term1}
        \mathfrak{T}_{t,1} (\by) \leq
        d^4 \, (d!)^2 J_{\star}^{-3} J_{\sup}^{2d} \tilde{J}_{\sup}^2  | \hv_h (\by) |^2.
    \end{equation}
    Plugging \eqref{eq:grad.piola.map.proof.term2.3} and \eqref{eq:grad.piola.map.proof.term1} into \eqref{eq:grad.piola.map.proof.init} yields the desired result.
\end{proof}
\end{document}